\documentclass[11pt,reqno]{amsart}
\usepackage[T1]{fontenc}
\usepackage[UKenglish]{babel}
\usepackage[margin=20mm]{geometry}

\usepackage{color}
\usepackage{bbm,bm}
\usepackage{mathtools} 
\usepackage{extarrows} 
\usepackage{stackrel}
\usepackage[shortlabels]{enumitem}
\usepackage{graphicx}
\usepackage{subfig}

\usepackage{amsmath,amssymb,amsfonts,amsthm}
\usepackage{mathrsfs,dsfont}
\usepackage{hyperref,url}
\usepackage{amssymb, amsmath, amsthm, amscd}
\usepackage{color}
\theoremstyle{plain}

\newtheorem{theorem}{Theorem}[section]
\newtheorem{lemma}[theorem]{Lemma}
\newtheorem{proposition}[theorem]{Proposition}
\newtheorem{corollary}[theorem]{Corollary}
\theoremstyle{remark}
\newtheorem{definition}[theorem]{Definition}

\newtheorem{remark}[theorem]{Remark}
\newcommand{\rmA}{{\mathrm{A}}}
\newcommand{\rmF}{{\mathrm{F}}}
\newcommand{\aAn}{{ \mathrm{A}_N(e)}}

\newcommand{\aAr}{{ \mathrm{A}_{R_e}(e)}}

\newcommand{\sO}{\mathbb{G}}

\newcommand\sC{{\mathscr C}}

\newcommand\nt{{\rmD^*_n}}
\newcommand{\rmD}{{\mathrm{D}}}
\newcommand{\rmT}{{\mathrm{T}}}
\newcommand{\rmd}{{\mathrm{d}}}
\newcommand\et{\rmT}
\newcommand\ega{{\gamma}}

\newcommand\mf{{\rmF_m}}

\newcommand\ti{{\rmT_n}}
\newcommand\ent{{\operatorname{Ent}}}

\newcommand\cE{{\kE(\Z^d)}}

\newcommand\pp{{\mathbb P}}

\newcommand{\pr}{\mathbb{P}}
\newcommand\cA{{\mathcal A}}
\newcommand\cB{{\mathcal B}}
\newcommand{\kB}{\mathcal{B}}
\newcommand\kC{{\mathcal C}}
\newcommand\kD{{\mathcal D}}
\newcommand\cF{{\mathcal F}}

\newcommand\kP{{\mathcal P}}
\newcommand\kE{{\mathcal E}}

\newcommand\cU{{\mathcal U}}
\newcommand\cV{{\mathcal V}}
\newcommand{\kV}{\mathcal{V}}
\newcommand\cW{{\mathcal W}}

\newcommand\cO{{\mathcal O}}

\newcommand\cL{{\mathcal L}}

\newcommand\N{{\mathbb N}}

\newcommand\Z{{\mathbb Z}}
\newcommand\R{{\mathbb R}}

\newcommand\E{{\mathbb E}}
\newcommand\I{{\mathbb I}}
\newcommand\bO{{\mathbb O}}

\newcommand{\var}{\operatorname{Var}}

\newcommand\e{{\mathbf e}}

\newcommand\clo{{\mathrm{clo}}}
\newcommand\ba{{\textbf{a}}}

\DeclareMathOperator{\Diam}{Diam}

\DeclareMathOperator{\Var}{Var}

\def\be#1{\begin{equation*}#1\end{equation*}}
\def\ben#1{\begin{equation}#1\end{equation}}
\def\bea#1{\begin{eqnarray*}#1\end{eqnarray*}}

\def\ba#1{\begin{align*}#1\end{align*}}
\def\ban#1{\begin{align}#1\end{align}}

\begin{document}
\allowdisplaybreaks
\title[Subdiffusive concentration for the chemical distance in Bernoulli percolation]
{Subdiffusive concentration for the chemical distance in Bernoulli percolation}
\author{Van Hao Can}
\address[Van Hao Can]{Institute of Mathematics, Vietnam Academy of Science and Technology, 18 Hoang Quoc Viet, Cau Giay, Hanoi, Vietnam} \email{cvhao@math.ac.vn}
\thanks{\emph{V.H. Can:}
Institute of Mathematics, Vietnam Academy of Science and Technology, 18 Hoang Quoc Viet, Cau Giay, Hanoi, Vietnam}
\author{Van Quyet Nguyen}
  \address[Van Quyet Nguyen]{
Institute of Mathematics, Vietnam Academy of Science and Technology, 18 Hoang Quoc Viet, Cau Giay, Hanoi, Vietnam} 
\email{nvquyet@math.ac.vn}
\thanks{\emph{V.Q. Nguyen:}
Institute of Mathematics, Vietnam Academy of Science and Technology, 18 Hoang Quoc Viet, Cau Giay, Hanoi, Vietnam}
\date{}
\maketitle
\begin{abstract}
 Considering supercritical Bernoulli percolation on $\Z^d$, Garet and Marchand \cite{garet2009moderate} proved a diffusive concentration for the graph distance.  In this paper, we sharpen this result by establishing the subdiffusive concentration inequality, which revisits the sublinear bound of the variance proved  by Dembin \cite{dembin2022variance} as a consequence.  Our approach is inspired by similar work in First-passage percolation \cite{benaim2008exponential, damron2014subdiffusive},  combined with new tools to address  the challenge posed by the infinite weight of the model. These tools, including the notion of effective radius and its properties, enable a simple one-step renormalization process as a systematic method of managing the effects of resampling edges. 
\end{abstract}
\allowdisplaybreaks
\section{Introduction}
\subsection{Model and main result} Bernoulli percolation is a simple  but well-known probabilistic model for  porous material introduced by Broadbent and Hammersley \cite{broadbent1957percolation}. Let $d\geq 2$ and $\cE$ be the set of the edges $e = ( x,y)$ of endpoints $x = (x_1,\ldots,x_d ), y = (y_1,\ldots,y_d) \in \Z^d$ such that $\|x - y\|_1 := \sum_{i=1}^d|x_i-y_i| = 1$. Given the parameter $p \in(0,1)$, we let each edge $e \in \cE$ be \emph{open} with probability $p$ and  \emph{closed} otherwise,  independently of the state of other edges. The phase transition of the model has been well-known since 1960s. In particular, it has been shown that there exists a critical parameter  $p_c(d) \in (0,1)$ such that  there is  almost surely a unique infinite open cluster $\kC_\infty$ if $p>p_c(d)$, whereas all open clusters are finite if $p<p_c(d)$. We refer  to  the book of Grimmett \cite{grimmett1999percolation} and the reviews of Duminil-Copin \cite{duminil2018sixty} for classical results and recent developments in percolation theory. In this paper, we are interested in the chemical distance that is the graph distance of infinite cluster in the supercritical regime.

 For each $A,B,U \subset \Z^d$, we define
 \begin{align*}
      \rmD^{U} (A,B) 
      & := \inf \{ |\gamma|: x \in A,\, y\in B,\,\gamma \text{ is an open nearest-neighbor path from } x \text{ to } y \text{ inside } U\}.
 \end{align*}
 When $U = \Z^d$, we simply write $\rmD$ for $\rmD^{\Z^d}$ and we write $\rmD(x,y)$ for $\rmD(\{x\},\{y\})$.  Since this graph  distance between two disconnected vertices is not well defined, we consider a regularized version as follows.  Let $x \in \Z^d$, we denote by $x^*$ the closest point to $x$ in $\kC_\infty$ (in $\| . \|_{\infty}$ distance) with a deterministic rule breaking ties, and call it the regularized point of $x$. The graph distance is then defined as: for $ x,y \in \Z^d$, 
\begin{align*}
 \rmD^*(x,y) := \rmD(x^*,y^*) = \inf \{ |\gamma|: \gamma \text{ is an open nearest-neighbor path from } x^* \text{ to } y^* \}. 
\end{align*}
Let $\e_1 = (1,0,\ldots,0)$ and we aim to study the asymptotic behavior of 
    \begin{align*}
    \rmD^*_n := \rmD^*(0,n \e_1).
\end{align*}
 The linear growth of $\rmD^*_n$ was described by Garet and Marchand \cite{garet2004asymptotic}: for any $p > p_c(d)$, there exists a constant $\mu_p(\e_1) \in (0,\infty) $ such that,  
\begin{align}\label{time constnat of D*}
    \lim_{n \to \infty} \dfrac{\rmD^*_n}{n} = \mu_p(\e_1) \quad \text {a.e and in } L^1.
\end{align}
The value  $\mu_p(\e_1)$ is called  the \textit{time constant}. The regularity of the time constant (Lipschitz-continuity of the time constant with respect to $p$) was proved in \cite{cerf2022time} and revisited in \cite{can2023lipschitz}. Naturally, the next question we are interested in is the fluctuation and deviation of the graph distance. The large deviation of  $\rmD^*_n$ has been well studied since 1990s, see, for example, \cite{antal1996chemical,garet2007large,dembin2022upper}. The moderate deviation of $\rmD^*_n$ (or precisely the concentration with diffusive scale) was established  by Garet and Marchand (see \cite[Theorem 1.2]{garet2009moderate}), using concentration inequalities of Boucheron, Lugosi, and Massart: for each $C > 0$, there exists some constant $c>0$ such that for all $\lambda \in [C \log n, \sqrt{n}$],
\begin{align} \label{mdv}
    \pr[|\rmD^*_n - \E[\rmD^*_n]| \geq \sqrt{n} \lambda] \leq c^{-1} \exp(-c \lambda).
\end{align}
Recently,  Dembin gave a sublinear bound on the variance of $\rmD^*_n$ (see \cite[Theorem 1.1]{dembin2022variance}): there exists a positive constant $C$ such that 
\ben{ \label{subb}
\Var (\rmD^*_n) \leq C \frac{n}{ \log n}.
}
The main result of our paper is to sharpen the moderate deviation \eqref{mdv} by establishing  a sub-diffusive concentration of $\rmD^*_n$. 
\begin{theorem}\label{mainthm}
Let $p > p_c(d)$. There exists a  positive constant $ c> 0$ such that for all  $n \in \N_{\geq 2} $ and $\kappa \geq 0$,
\begin{align}\label{7}
\mathbb{P}\left(|\rmD^*_n-\mathbb{E} [\rmD^*_n]| \geq \sqrt{\frac{n}{\log n}} \kappa\right) \leq c^{-1} \exp(-c \kappa).
\end{align}
Consequently, the sublinear bound of the variance \eqref{subb} holds. 
\end{theorem}
\begin{remark}
If $n$ is small and $\kappa < 1$, we can choose $c$ sufficiently small so that \eqref{7} holds trivially. Throughout the rest of this paper, we therefore focus on the case where $n$ is sufficiently large and $\kappa \geq 1$.
\end{remark}
It is worth  mentioning that  the present work is inspired by similar result in First-passage percolation (FPP) by Damron, Hanson, and Sosoe. We briefly recall here the FPP model and related results. Let $(t_e)_{e \in \cE}$ be i.i.d. random variables with common  distribution $\zeta$. Each edge $e\in \cE$ is assigned a random weight $t_e$. Then the first passage time is defined in the same way as the chemical distance.  Precisely, if  $\zeta([0,\infty))=1$, i.e., the edge weights $(t_e)_{e \in \cE}$ are finite: for $\{x,y\} \subset U \subset \Z^d$,
\begin{align}\label{T}
   \rmT^U(x,y) := \inf_{\substack{\gamma:x \to y \\ \gamma \subset U}} \sum_{e \in \gamma}t_e,
\end{align}
where infimum is taken over the set of nearest-neighbor paths inside $U$ from $x$ to $y$. When $U = \Z^d$, we simply write $\rmT$ instead of $\rmT^{\Z^d}$. If $p_c(d) < \zeta([0,\infty))<1$, we set 
\begin{align}\label{T*}
   \rmT^*(x,y) := \rmT(x^*,y^*),
\end{align}
where $z^*$ denotes the closest point to $z$ in the infinite cluster of finite weight edges with a deterministic rule to break ties.  Then   the supercritical Bernoulli percolation   can be referred  as a  particular case of FPP with  distribution $ \zeta=\zeta_p= p \delta_1 +(1-p) \delta_{\infty}$.

The convergence of the scaled passage time in probability to time constant was obtained by  Cerf and Théret \cite[Theorem 4]{cerf2016weak}, without any moment assumption: there exists a constant $\mu_{\zeta}(\mathbf{e}_1) \in [0,\infty)$ such that
\begin{align}\label{limit}
       \lim_{n \to \infty} \frac{\rmT^*(0, n\mathbf{e}_1)}{n} =  \mu_{\zeta}(\mathbf{e}_1) \qquad \text{   in  probability}.
\end{align}
In \cite[Remark 1]{garet2004asymptotic}, Garet and Marchand proved that if $\E[t_e^{2+\varepsilon} \I(t_e < \infty)] < \infty$ with some $\varepsilon >0$, then the convergence in \eqref{limit} holds true almost surely and in $L^1$.  The continuity property of the time constant with respect to $\zeta$ was first shown by Cox and Kesten in \cite{cox1980time,cox1981continuity, kesten1986aspects} for FPP with finite weights. This was later extended for FPP with
possibly infinite weights by Garet, Marchand, Procaccia, and Théret in \cite{garet2017continuity}. It then has been proved by Damron, Hanson, and Sosoe in \cite{damron2015sublinear} that if $\zeta([0,\infty))=1$ and $\E[t_e^2 \log_+t_e ] < \infty$, then the sublinear variance holds: 
\ben{\label{super}
\Var(\rmT(0, n\e_1)) \leq C \frac{n}{\log n}.
}
 The sublinearity of variance is a particular case of superconcentration, a  phenomenon coined by Chatterjee \cite{chatterjee2014superconcentration} describing the situation that usual concentration tools such as Poincar\'e or Efron--Stein inequality give suboptimal bounds.  In fact, in some particular models as Gaussian disordered systems \cite{chatterjee2014superconcentration},  he also established a connection between superconcentration and a chaotic phenomenon of the ground states.  However, understanding these notions beyond the Gaussian realm remains quite limited. Recently, Ahlberg, Deijfen, and Sfragara \cite{ahlberg2023chaos} showed that  this deep relation holds true in the context of First-passage percolation when the weight distribution $\zeta$ has a finite second moment. 
 Other examples of the superconcentration phenomena can be found in \cite{bernstein2020sublinear, chatterjee2023superconcentration, dembin2024superconcentration}.
\subsection{Method of the proof}\label{medthodofproof}
We first  remark that in supercritical Bernoulli percolation, the moment conditions of FPP mentioned above do not hold any more, since $\zeta(\{\infty\})=1-p>0$.   Hence, we need new tools to handle the infinite weight.
  
The general strategy of proving the subdiffusive concentration for FPP has  been shown in  \cite{benaim2008exponential} and  \cite{damron2014subdiffusive}. First, by general bounds, the subdiffusive concentration of $\rmT(0,n\e_1)$ can be reduced to some estimates of the variance of exponential functionals of $\rmT(0,n\e_1)$.  Next, to bound the variance of these exponential terms, one could use the entropy inequalities  combining with the  geometric averaging trick of Benjamini, Kalai, Schramm introduced in their highly influential paper \cite{benjamini2003improved}. Moreover, in FPP \cite{damron2014subdiffusive}, the key for variance or entropy  bounds  is to effectively control for the effect of resampling an edge weight, and the moment conditions play a   crucial role in these controls. However, in the context of the graph distance in Bernoulli percolation, closing an edge on the geodesic can have a significant impact on the graph distance due to the possibility of infinite edge-weight values.  By introducing the notion \textit{effective radius} and exploiting the \textit{greedy lattice animals}, we provide a simple one-step renormalization process as a systematic way to manage the effect of resampling edges. Let us now explain our proof in more detail. \\
 
\noindent \textbf{Part A: Subdiffusive concentration for the truncated first passage time.} We recall that the Bernoulli percolation can be viewed as a first passage percolation with the edge weights distribution $ \zeta_p= p \delta_1 +(1-p) \delta_{\infty}$. Let $(t_e)_{e \in \kE(\Z^d)}$ be a collection of i.i.d. random variables with the commond distribution as 
\be{
t_e \sim  p \delta_1 +(1-p) \delta_{\log^2 n}.
}
We refer to these as the weights truncated at level $\log^2 n$. We denote by
\[
\rmT_n:=\rmT(0,n\mathbf{e}_1),
\]
the truncated first passage time from from $0$ to $n\mathbf{e}_1$ associated with the truncated weights $(t_e)_{e\in \mathcal E(\mathbb Z^d)}$. Our aim is to prove the following.
\begin{theorem}\label{subdiffconforT}
There exists a positive constant $ c$ such that for all $n \in \N_{\geq 2}$ and $\kappa  \geq 0$,
\begin{align}\label{subtn}
\mathbb{P}\left(|\ti-\mathbb{E} [\ti]| \geq \sqrt{\tfrac{n}{\log n}} \kappa\right) \leq c^{-1} \exp(-c \kappa ).
\end{align}
\end{theorem}
 
Our argument initially follows the common scheme as in \cite{benaim2008exponential, damron2014subdiffusive}. Considering $\rmF_m$  a spatial  average of $\rmT_n$ defined in \eqref{mf} (inspired by Benjamini, Kalai, and Schramm in \cite{benjamini2003improved}), the proof of \eqref{subtn} can be transfered to the following. There exists a constant $c>0$ such that with $K = \dfrac{c n}{\log n}$,
 \begin{align} \label{variancebound}
    \forall \, |\lambda| < \dfrac{1}{\sqrt{K}} ,\quad \var\left(e^{\lambda \mf}\right) \leq K \lambda^2 \E\left[e^{2\lambda \mf}\right] < \infty.
 \end{align}  
  Using Falik-Samorodnitsky inequality (Lemma \ref{fs}), the problem \eqref{variancebound} is then reduced to understand well the behavior of  $G= e^{\lambda \rmF_m}$ when flipping the state of edges in $\cE$. In fact, we shall show in Section \ref{sec5} that the above variance bound is closely related to a large deviation estimate
 \ben{ \label{bore}
 \pp \Bigg( \sum_{e \in \kE(\Z^d)} (\nabla_e \rmT_n)^2 \geq Cn \Bigg) \leq \exp(- \sqrt{n} (\log n)^C), \qquad \nabla_e \rmT_n = \rmT_n(\log^2 n, t_{e^c}) - \rmT_n(1, t_{e^c}),
 }
where C is a suitable positive constant, and the \textit{discrete derivative} $\nabla_e \rmT_n$ measures the difference of $\rmT_n$ when flipping the state of $e$. 
 
 In order to get a suitable bound for $\nabla_e \rmT_n$, we introduce a novel notion called \textbf{\textit{effective radius}}. Roughly speaking, the effective radius of an edge $e$ is the smallest radius $R_e$ ensuring a good path bypassing $e$ inside the annulus $\rmA_{R_e}(e)= \Lambda_{3R_e}(e)\setminus \Lambda_{R_e}(e)$, and hence it measures the effect of the edge $e$ to  $\rmT_n$ as
 \ben{ \label{bonab}
 |\nabla_e \rmT_n| \leq \hat{R}_e \I(e \in \gamma), \quad  \hat{R}_e=\min \{C_* R_e, \log^2 n\},
 }
 with $C_*$ being a constant introduced in \eqref{C*} and $\gamma$ is a geodesic of $\rmT_n$. While the bound by $C_* R_e$ follows from using the bypass in the definition of effective radius, the bound by $\log^2 n$ is straightforward from the definition of the truncated passage time.  We also show in Section \ref{seckey} that the effective radii have two important properties.
 \begin{itemize}
     \item [(a)] First, they are weakly dependent in the sense that the event $\{R_e \leq t \}$ depends only on the status of edges at a distance at most $C_*t$ from $e$.  
     \item[(b)]  Second, they exhibit the finite range exponential decay: $\pp(R_e \geq t) \leq \exp(-ct)$ for all $1 \leq t \leq \exp(c \log^2 n)$ with $c$ a  positive constant. 
 \end{itemize}
Next, we will need an ingredient, the so-called \textbf{\textit{greedy lattice animals}} (see Section \ref{sec3} for more details), to control the total cost $\sum_{e \in \gamma} \hat{R}_e^2$. Precisely, we decompose 
\ben{ \label{decom}
\sum_{e \in \gamma} \hat{R}_e^2 = \sum_{M=1}^{ \lfloor (\log n)^2 \rfloor +1} M^2 \sum_{e \in \gamma} \I(M-1\leq \hat{R}_e <M),
}
and exploit the greedy lattice animals theory (see \cite{damron2014subdiffusive, nakajima2019first})  to the Bernoulli random variables $(\I(M-1 \leq \hat{R}_e<M))_{e \in \kE(\Z^d)}$ with the aid of two properties (a) and (b). In fact, we obtain a much better estimate than \eqref{bore}: 
\ben{ \label{opldp}
 \pp \Bigg( \sum_{e \in \kE(\Z^d)} (\nabla_e \rmT_n)^2 \geq Cn \Bigg) \leq \pp \Bigg( \sum_{e \in \gamma } \hat{R}_e^2 \geq Cn \Bigg)  \leq \exp(- n/ (\log n)^C).
}
 We emphasize that the two properties (a) and (b) alone are not strong enough to gain the above large deviation estimate, and in fact
 the truncation by $\log^2 n$ is crucial for this purpose, as detailed in Lemma \ref{lemkey}. 
 \\

\noindent \textbf{Part B: Bound on discrepancy between $\rmD^*_n$ and $\rmT_n$ via effective radii}. We aim to prove that $|\rmD^*_n -\rmT_n|$ is negligible with overwhelming probability.
\begin{theorem}\label{prodecre}
 There exists a positive constant $c$ such that for all $L \geq \log^2 n$, 
\begin{align}\label{45}
\pr\left(|\nt - \rmT_n| \geq L \right) \leq  c^{-1} \exp\left(-c \tfrac{L}{\log L + \log^2 n }\right).
\end{align}
\end{theorem}
We reduce the problem to comparing $\rmD^*_n $ and $\rmT(0^*,(n\mathbf{e}_1)^*)$, since the discrepancy between $\rmT_n = \rmT(0,n\mathbf{e}_1)$ and $\rmT(0^*,(n\mathbf{e}_1)^*)$ is negligible by the fact that the regularized points are close to their original points. Therefore, the heart of the proof of this theorem  is to make a suitable coupling of $\rmD^*_n $ and $\rmT(0^*,(n\mathbf{e}_1)^*)$.   Pick a geodesic of $\rmT(0^*,(n\mathbf{e}_1)^*)$, say $\gamma$. First, we identify  $e_*$ the closed (or $\log^2n$-weight) edge with largest effective radius and find a bypass for it. Hence, we get a new path, say $\eta$,  from $0^*$ to $(n\textbf{e}_1)^*$ with fewer closed edges than the original path $\gamma$ and $|\rmT(\gamma)-\rmT(\eta)| =\cO(R_{e_*})$. Next, apply the same procedure to the new path $\eta$. We continue this covering process until there are no closed edges left. When this process ends,  we find a set of closed edges $\Gamma\subset \gamma$ such that
 \begin{itemize}
     \item [(i)]  $\forall e, e' \in \Gamma, \|e-e'\|_{\infty} \geq \max\{R_e,R_{e'}\}$, 
     \item [(ii)] $|\nt - \rmT(0^*,(n \e_1)^*)| \leq C \sum_{e \in \Gamma} R_e$,
\end{itemize}
with $C$ a positive constant.  Hence, our task is now to show 
\ben{ \label{sreG}
\pp \Bigg( \sum_{e \in \Gamma}  R_e  \geq L \Bigg) \leq c^{-1}\exp\left(-c \tfrac{L}{\log L + \log^2 n}\right).
}
As mentioned above, achieving a large deviation of the total cost $\sum_{e \in S} R_e$ with $S$ a general set, as described in \eqref{sreG}, is not feasible. In fact, the separable property (i) somehow  strengthens the local dependent of $(R_e)_{e \in \Gamma}$, and enables us to prove this large deviation estimate.  We refer to Section \ref{sec6} for the detailed proof. \\

We conclude by emphasizing the novel aspects of our approach. While the construction of detours that bypass a given edge with controlled length is a natural and commonly employed idea, existing methods often rely on intricate multi-step renormalization schemes. These typically involve constructing bypasses through sequences of good boxes across multiple scales, using a random number of boxes with different but deterministic radii. In contrast, our method introduces a simplified one-step renormalization procedure, utilizing a single box with a random radius. This streamlined construction enables us to derive a nearly optimal large deviation estimate \eqref{opldp}. Intuitively, we encapsulate the multi-scale analysis within the decomposition \eqref{decom} and leverage the greedy lattice animals theory to analyze each scale. Our method exhibits robustness and is potentially applicable to a broader class of problems involving the control of resampling effects in percolation, such as the Lipschitz continuity of the time constant, as discussed in \cite{can2023lipschitz}.

\subsection{Proof of the main result assuming Theorems \ref{subdiffconforT} and \ref{prodecre}} Using Theorem \ref{prodecre},
\begin{align*}
     \E[|\nt - \rmT_n|] = \cO(\log^2 n).
\end{align*}
By the triangle inequality,
\begin{align*}
|\rmD^*_n-\E[\nt]| \leq |\rmD^*_n-\ti|+ |\ti-\E[\ti]|+|\E[\ti]-\E[\nt]|.
\end{align*}
It follows from the last two estimates that for all $\kappa \geq 1$ and $n$ large enough,
\begin{align*}\label{divesub}
    \pr\Big[|\nt - \E[\nt]|\geq \kappa \sqrt{\tfrac{n}{\log n}}  \Big] & \leq \pr\Big[|\rmT_n - \E[\rmT_n]|\geq \tfrac{\kappa}{4} \sqrt{\tfrac{n}{\log n}}  \Big] + \pr\Big[|\rmD^*_n - \rmT_n| \geq \tfrac{\kappa}{4} \sqrt{\tfrac{n}{\log n}}  \Big].
\end{align*}
By Theorem \ref{subdiffconforT},
\be{
    \pr\Big[|\rmT_n - \E[\rmT_n]|\geq \tfrac{\kappa}{4} \sqrt{\tfrac{n}{\log n}}  \Big] \leq c_1^{-1} \exp(-c_1 \kappa/4),
}
for some $c_1 >0$. In addition, using Theorem \ref{prodecre}, 
\begin{align*} 
\pr\Big[|\rmD^*_n - \rmT_n|  \geq \tfrac{\kappa}{4} \sqrt{\tfrac{n}{\log n}}  \Big] &  \leq c_2^{-1} \exp\Big(-c_2 \kappa \tfrac{\sqrt{n}}{(\log n)^{5/2}}\Big),
\end{align*} 
for some $c_2 >0$. Finally, combining  the last three displays we get Theorem \ref{mainthm}. \qed 
\subsection{Organization and notation of this paper}
The paper is organized as follows. In Section \ref{seckey}, we present the construction of random effective radius and its application to control the effect of flipping an edge. We study some moments and large deviations of lattice animal of dependent weight in Section \ref{sec3}. In Section \ref{sec5}, we first revisit the concentration inequalities and then prove the subdiffusive concentration of the modified graph distance (Theorem \ref{subdiffconforT}). Finally, we estimate the discrepancy between the graph distance and its modified version via the covering argument (Theorem \ref{prodecre}) in Section \ref{sec6}.

 To conclude this section, we introduce some notations used in the paper. 
 \begin{itemize}
   \item \emph{Integer interval}. Given an integer $t \geq 1$, we denote by $[t]:= \{1, 2,\ldots,t\}$.
  \item \emph{Metric}. We denote by $\|\cdot \|_{1},\|\cdot \|_{\infty}, \|\cdot\|_{2}$  the $l_1,l_{\infty},l_{2}$ norms, respectively.  
  \item \emph{Box and its boundary}. Let $x \in \Z^d$ and $t \geq 1$, we will denote by $\Lambda_t(x):= x+ [-t, t]^d \cap \Z^d$ the box centered  at $x =(x_1,\ldots,x_d) \in \Z^d$ with radius $t$. For convenience, we briefly write $\Lambda_t = [-t, t]^d \cap \Z^d $ for $\Lambda_t(0)$. We define the boundary of $\Lambda_t(x)$ as $\partial \Lambda_t(x):=\Lambda_t(x) \setminus \Lambda_{t-1}(x)$.
   \item \emph{Set of boxes}. For $m \in \N$ and $X \subset \Z^d$, let $\cB_m(X)$ denote the set of all boxes of side-length $m$ in $X$.
   \item \emph{Set distance}. For $X,Y \subset \Z^d$, we denote ${\rm d}_{\infty}(X,Y)$ the distance between $X$ and $Y$ by
     \begin{align*}
         {\rm d}_\infty(X,Y) :=  \inf \{\|x-y\|_\infty:x \in X,y \in Y \}.
     \end{align*}
     \item \emph{Edge distance}. For each edge $e \in \cE $, we pick a deterministic rule to represent $e= (x_e,y_e)$, for example, $\|x_e\|_{1}< \|y_e\|_1$. For $x\in \Z^d$ and $e,f \in \cE$, we denote by
     \begin{align*}
         \|e-f\|_{\infty} := \|x_e -x_f\|_{\infty} \quad \|e-x\|_{\infty}:= \|x_e- x\|_{\infty}.
     \end{align*}
   \item \emph{$\Z^d$-path and set of $\Z^d$-paths}. For any $\ell$, we say that a sequence $\gamma = (v_0,\ldots, v_\ell)$ is a $\Z^d$-path if for all $i \in [\ell],  \|v_i - v_{i-1}\|_1 =1$. The length of $\gamma$ is $\ell$, denoted by $|\gamma|$. For $1 \leq i<j \leq \ell$, we denote by $\gamma_{v_{i},v_{j}}$ the subpath of $\gamma$ from $v_{i}$ to $v_j$. In addition, if $v_i \neq v_j$ for $i \neq j$, then we say that $\gamma$ is self-avoiding. \textit{From now on, we will shortly write a path in place of a self-avoiding $\Z^d$-path}. For any path $\gamma$, we denote ${\bf s}(\gamma),{\bf e}(\gamma)$ the starting and ending vertices of $\gamma$, respectively. Given $U \subset \Z^d$, let $\mathcal{P}(U)$ be the set of all paths in $U$.
\item \emph{Paths concatenation}. Given two paths $\gamma^1,\gamma^2$ such that ${\bf e}(\gamma^1)= {\bf s}(\gamma^2)$, we denote $c(\gamma^1,\gamma^2)$ the first point of $\gamma^1$ (in order from ${\bf s}(\gamma^1)$ to ${\bf e}(\gamma^1)$)  that intersects with $\gamma^2$ and so $\gamma^1_{{\bf s}(\gamma^1),c(\gamma^1,\gamma^2)} \cap \gamma^2_{c(\gamma^1,\gamma^2),{\bf e}(\gamma^2)} =\{c(\gamma^1,\gamma^2)\} $. Notice that if $\gamma^1 \cap \gamma^2 = \{{\bf e}(\gamma^1)\} = \{{\bf s}(\gamma^2)\}$ then $c(\gamma^1,\gamma^2) = {\bf e}(\gamma^1)= {\bf s}(\gamma^2)$. We denote the paths concatenation of $\gamma^1$ and $\gamma^2$ by
\begin{align*}
    \gamma^1 \oplus \gamma^2 := ({\bf s}(\gamma^1), \ldots, c(\gamma^1,\gamma^2),\ldots,{\bf e}(\gamma^2)).
\end{align*}
 \item \emph{Open path, open cluster and crossing cluster}. Given a Bernoulli percolation on $\Z^d$ with parameter $p$, let $\mathcal{G}_p = (\Z^d, \{ e \in \cE: e \text{ is open} \})$. We say that a path is open if all of its edges are open. An open cluster is a maximal connected component of $\mathcal{G}_p$. An open cluster $\kC$ crosses a box $\Lambda$, if for all $d$ directions, there is an open path in $\kC \cap \Lambda$ connecting the two opposite faces of $\Lambda$. 
 \item  \emph{Open connection}. Given $A,B,U \subset \Z^d$, we write $ A \xleftrightarrow{ U} B$ if there exists an open path inside $U$ connecting $A$ to $B$; otherwise, we write $A \not \xleftrightarrow{ U} B$. 
 When $U = \Z^d$, we omit the symbol $\Z^d$ for simplicity.
\item \emph{Diameter}. For $A \subset \Z^d$ and $1 \leq i \leq d$, let us define
$$
\Diam_i(A)= \max_{x,y \in A}|x_i-y_i|,
$$
 and we thus denote $\Diam(A)$ the diameter of $A$ by 
\begin{align*}
     \Diam(A) = \max_{1 \leq i \leq d} \text{Diam}_i(A).
 \end{align*}
\item  \emph{Geodesic}. Given $\{x,y\} \subset U \subset \Z^d$, we denote by $\gamma$ a geodesic between $x$ and $y$ of $\rmD^U(x,y)$ if $
\gamma$ is an open path inside $U$ such that $|\gamma| = \rmD^U(x,y)$. We also denote by $\eta$ a geodesic between $x$ and $y$  of $\rmT^U(x,y)$ if $\eta$ is a path inside $U$ such that $\rmT(\eta) = \rmT^U(x,y)$. If there are several choices for $
\gamma$ or $\eta$, we choose one of them according to a deterministic rule to break ties.
  \end{itemize}
  
\section{The effect of resampling}\label{seckey}

As mentioned in Introduction, the key to the concentration of the truncated first passage time $\rmT_n$ is understanding the cost of resampling the edge weights along the geodesic of $\rmT_n$. To study this issue, given an edge $e$, we introduce the notion of \textit{effective radius} $R_e$, which measures how large the change in passage time is when flipping the weight of $e$. The radius $R_e$  indeed guarantees an open bypass   of $e$ whose length is comparable to $R_e$. Hence, it turns out that the cost of flipping the edge weight is given by \eqref{bonab}. We then investigate the properties of the effective radii, including the local dependence and light-tailed decay distribution. These properties will play a central role in controlling the total cost of resampling edges.
\subsection{Effective radius and its applications} \label{effective radius and its app}
 We couple Bernoulli percolation with First-passage percolation where $\zeta = p \delta_1 +(1-p) \delta_{\log^2 n}$ as follows: each open edge $e$ has weight $1$ and $\log^2 n$ otherwise. We now set up some definitions for the notion of \textit{effective radius}. Recall that $\kP(A)$ is the set of all paths in $A$. Let us define the set of open paths in $A\subset \Z^d$ by 
 \be{
  \mathbb{O}{(u,v;A)}:=\{ \gamma \in \kP(A): {\bf s}(\gamma)= u, {\bf e}(\gamma)= v, \textrm{ $\gamma$ is open}\};\quad \mathbb{O}{(A)}:= \bigcup_{u,v \in A}  \mathbb{O}{(u,v;A)},
 }
 here recall that ${\bf s}(\gamma)$ and ${\bf e}(\gamma)$ are starting and ending points of $\gamma$. If $A = \Z^d$, we simply write $\mathbb{O}$ and $\mathbb{O}(u,v)$ for $\mathbb{O}(\Z^d)$ and $\mathbb{O}(u,v;\Z^d)$, respectively. For $u,v \in A\subset \Z^d$, we define the set of geodesics of  $\rmT^A(u,v)$ as
 \be{
 \sO(u,v;A) :=\{\gamma \in \kP(A): {\bf s}(\gamma)= u, {\bf e}(\gamma)=v, \rmT(\gamma)=\rmT ^{A}(u,v)     \},
 }
 and the set of geodesics in $A$ as
 \begin{align*}
     \sO(A) := \bigcup_{u,v \in A} \sO(u,v;A).
 \end{align*}
 Also, define the set of modified geodesics by
 \bea{
 \sO_*(u,v;A) &:=& 
 \Big \{\gamma \in \kP(A): {\bf s}(\gamma)= u, {\bf e}(\gamma)= v,\, \exists \xi =(\xi_e)_{e \in \cE} \text{ with } \xi_e = t_e \, \forall e \in A, \\
& & \qquad \qquad \qquad \qquad \qquad \qquad \, \exists \pi \in \sO^\xi(\Z^d) \text{ such that }\gamma \setminus \pi \textrm{ is open}  \Big\}, \\
  \sO_*(A) &:=&  \bigcup_{u,v \in A} \sO_*(u,v;A),
}
where $\sO^\xi(\Z^d)$ denotes the set of all geodesics in $\Z^d$ with respect to the configuration $\xi$. We adopt the convention that an empty path is open (particularly, $\mathbb{O}(A) \subset \sO_*(A), \sO(\Z^d) \subset \sO_*(\Z^d)$).
In other words, $\sO_*(u,v;A)$ consists of paths from $u$ to $v$ contained in $A$ that can be obtained by concatenating subpaths of a geodesic with open paths.
If $A=\Z^d$, we simply write $\sO(u,v), \sO, \sO_*, \sO_*(u,v)$ for $\sO(u,v;\Z^d)$, $\sO(\Z^d)$, $\sO_*(\Z^d)$, $\sO_*(u,v;\Z^d)$, respectively. 
\begin{remark}
 The set of modified geodesics $\sO_*$ arises from our goal of controlling the difference between the graph distance and the truncated first passage time. To achieve this, we build an inductive process that, step by step, replaces certain subgeodesics of the truncated first passage time with open paths (see Proposition \ref{lemdis} for more details). Consequently, at each step, the updated geodesic is no longer in $\sO$, but instead belongs to $\sO_*$.
\end{remark}
\begin{remark} \label{rem:omkx} 
Given $B \subset A \subset \Z^d$, if $\gamma \in \sO(A)$ (or $\gamma \in \sO_*(A)$) and $\gamma'$ is a subpath of $\gamma$ such that $\gamma' \subset B$, then $\gamma' \in \sO(B)$ (or $\gamma' \in  \sO_*(B)$). In addition,  $\sO(A)$ and  $\sO_*(A)$ are measurable with respect to edge weights inside $A$. Moreover, if  $\gamma \in \sO_*(u,v)$ and $\eta$ is a path from $u$ to $v$ such that $\eta \setminus \gamma$ is open, then $\eta \in \sO_*(u,v)$ as well.
 \end{remark}
  Recall that for each $e \in \cE$, we fix a deterministic rule to write $e=(x_e,y_e)$ so that $\|x_e\|_1 <\|y_e\|_1$. For each $N \geq 1$ and $e=(x_e,y_e)$, we define $\Lambda_N(e) := \Lambda_N(x_e)$ and the family of $N$-annuli:
\begin{align}
    \forall e \in \cE,\quad \aAn
    := \Lambda_{3N}(e) \setminus \Lambda_{N}(e), \quad \partial \aAn : = \partial \Lambda_{3N}(e) \cup \partial \Lambda_{N+1}(e).
\end{align}
 A path $\eta \subset \aAn$ is called a \textit{crossing path} of $ \aAn$ if it joins $\partial \Lambda_{N+1}(e)$ and $\partial \Lambda_{3N}(e)$. Let $\sC(\aAn)$ be the collection of all crossing paths of $\aAn$. Suppose that $\gamma$ is a path crossing the annulus $\rmA_N(e)$ at least once ($\gamma$ is not necessarily inside  $\rmA_N(e)$). Then, we denote by $\gamma^{{\bf i},e} \in \sC(\aAn)$ and $\gamma^{{\bf o},e} \in \sC(\aAn)$ the first and last subpaths of $\gamma$ crossing $\rmA_{N}(e)$, respectively, ordered from ${\bf s}(\gamma)$ to  ${\bf e}(\gamma)$ (see Figure \ref{Crp}-A).
\begin{figure}[htbp]
    \centering
    \hspace{-2 cm}
    \subfloat[\centering First and last crossing paths.]{{\includegraphics[width=8cm]{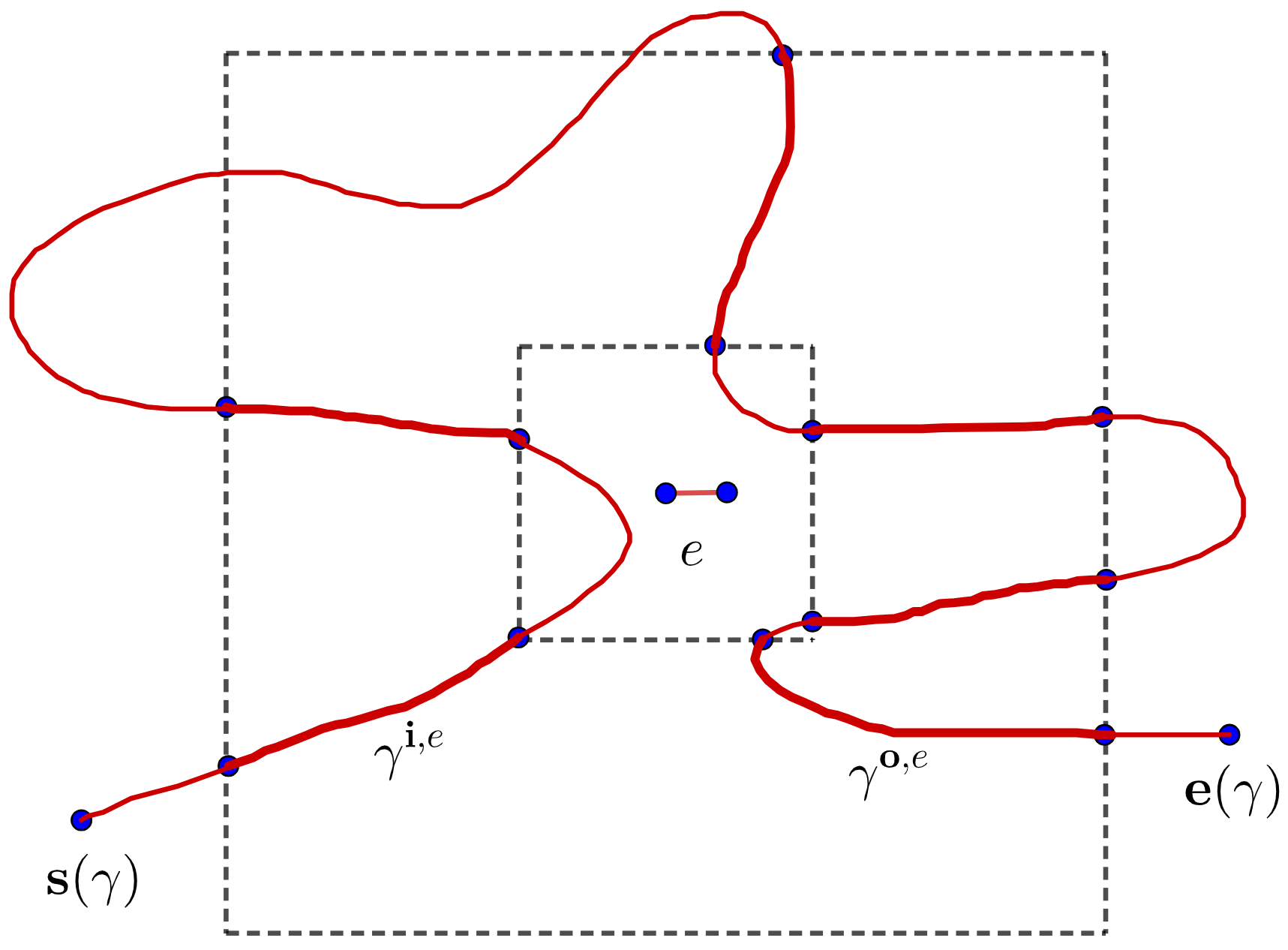} }}%
    \,
    \subfloat[\centering A short open path  $\pi_e$ joining crossing paths.]{{\includegraphics[width=5.8cm]{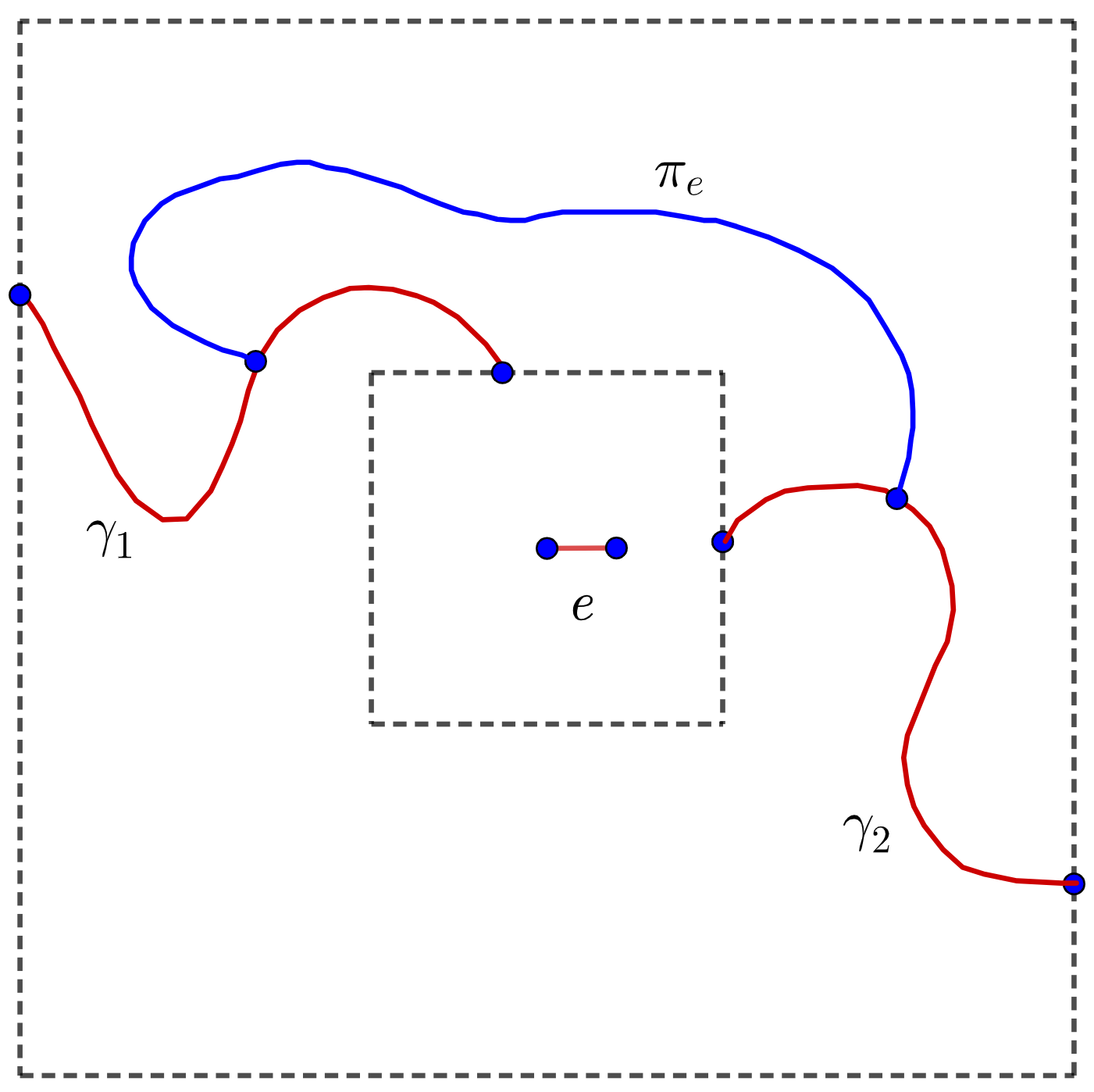} }}%
    \caption{Illustration of crossing paths  and the event $\kV^1_e$.}%
    \label{Crp}
\end{figure}
Fix $\rho$ and $c_*$ as in Lemmas \ref{Lem: large deviation of graph distance Dlambda} and \ref{lem: crossing cluster}, respectively, and define
\begin{align}\label{C*}
  C_* := 1/c_* +    (100 \rho^2 )^d;
\end{align}
and let
\begin{align*}
& \mathcal{V}_N^1(e):=\{\forall \, \gamma_1, \gamma_2 \in \sO_{*}(\Lambda_{C_*N}(e)) \cap \sC(\aAn): 
 \,\, \rmD^{\aAn}(\gamma_1,\gamma_2) \leq C_* N \}; \\
 & \mathcal{V}_N^2(e):=\{\forall \,  x,y \in \Lambda_{3N}(e)\text{ with } \rmD^{\Lambda_{3N}(e)}(x,y)< \infty: \,\, \rmD^{\Lambda_{4N}(e)}(x,y) \leq C_* N\}.
\end{align*} 
\begin{definition}\label{Def: effective radius}
    For each  $e \in \cE$,  we define the \textbf{effective radius} of $e$ as
\begin{align} 
    R_e := \inf\{N \geq 1: \mathcal{V}_N^1(e)\cap \mathcal{V}_N^2(e) \text{ occurs} \}.
\end{align}
\end{definition} 
 Roughly speaking, $R_e$ is the radius that 
 ensures the existence of an economically feasible bypass for any path in $\sO_*$
crossing the annulus with a cost comparable to $R_e$ (see Figure \ref{Crp}-B). Additionally, as we will see below, $R_e$ has beneficial properties such as local dependence and a light-tailed distribution, using the robust connectivity of the supercritical percolation.
\begin{remark}
In \cite{can2023lipschitz}, Nakajima and we also introduced a simplified version of the effective radius to study the Lipschitz continuity of the time constant, motivated by an earlier draft of this paper. Here, we highlight the key difference between these notions. In the present paper, the event  $\cV^1_N$ 
  involves the graph distance between two more complex paths, rather than focusing solely on geodesics as in \cite{can2023lipschitz}. This modification, along with the introduction of 
$\cV^2_N$, aims to better control the discrepancy between the graph distance and the truncated first-passage time. Notably, this change allows us to obtain exponentially large deviation estimates for the effective radius, rather than the polynomial bounds established in \cite{can2023lipschitz}. We emphasize that while the moment bounds for this discrepancy were sufficient in \cite{can2023lipschitz}, they are no longer adequate in the present paper.
\end{remark}
\begin{remark} \label{rem:weak} 
By the definition of effective radius and Remark \ref{rem:omkx}, for any $e \in \cE$ and $t \geq 1$, the event $\{R_e=t\}$ depends solely on the state of edges in the box $\Lambda_{C_*t}(e)$.
\end{remark}
 We can now state the large deviation estimate for these effective radii: 
\begin{proposition}\label{Lem: effective radius}
   There exists a constant $ c>0$ such that 
         for all $e \in \cE $ and $1 \leq t \leq \exp(c\log^2 n)$,
         \begin{align*}
             \pr (R_e \geq t)  \leq  c^{-1} \exp(-c t).
         \end{align*}
\end{proposition}
 The proof of Proposition \ref{Lem: effective radius} is provided in Section \ref{sec 2.2}. Given an edge $e$ in a modified geodesic, the following proposition allows us to build a detour that avoids $e$ with an economical cost that is comparable to the effective radius. 
\begin{proposition} \label{Lem: application of er}
Let $x,y \in \Z^d$ and $\gamma \in \sO_*(x,y)$, and $e$ be an edge in $\gamma$ such that $x,y \notin \Lambda_{3R_e}(e)$. Then there exists another path $\eta_e \in \sO_{*}(x,y)$ such that: 
         \begin{itemize}
             \item [(i)] $ \eta_e  \cap \Lambda_{R_e}(e)  = \varnothing$, and $\eta_e \setminus \gamma \subset \rmA_{R_e}(e)$, and $\eta_e \setminus \gamma$ is open;
             \item [(ii)] $|\eta_e \setminus \gamma| \leq C_*  R_e$; 
           \end{itemize}
\end{proposition}
\begin{proof}
    Since $e \in \gamma$ and $x,y \notin \Lambda_{3R_e}(e)$, $\gamma$ crosses the annulus $\aAr$ at least twice. Let $\gamma^{ {\bf i},e}$ and $\gamma^{ {\bf o},e}$ be these first and last crossing paths of $\aAr$, respectively. Then we have $\gamma^{ {\bf i},e}, \gamma^{ {\bf o},e} \in  \sC(\rmA_{R_e}(e))$ and $\gamma^{ {\bf i},e} \neq \gamma^{ {\bf o},e}$. By the hypothesis that $\gamma \in  \sO_{*}$ and Remark \ref{rem:omkx}, both $\gamma^{ {\bf i},e}$ and $ \gamma^{ {\bf o},e}$ belong to $ \sC(\rmA_{R_e}(e)) \cap \sO_{*}(\Lambda_{C_*R_e}(e))$. Furthermore, by the definition of $R_e$, the event $\mathcal{V}^1_{R_e}(e)$ occurs, and so $\rmD^{\aAr}(\gamma^{{\bf i},e}, \gamma^{ {\bf o},e}) \leq C_* R_e$. Let $\pi_e$ be a geodesic of $\rmD^{\aAr}(\gamma^{{\bf i},e}, \gamma^{ {\bf o},e})$. Then $\pi_e$ consists of only open edges and satisfies
\begin{align*}
   |\pi_e| =  \rmD^{\aAr}(\gamma^{{\bf i},e}, \gamma^{ {\bf o},e}) \leq C_* R_e.
\end{align*}
Suppose that  $\pi_e$ intersects with $\gamma^{ {\bf i},e}$ and $ \gamma^{ {\bf o},e}$ at $z_{{\bf i}}$ and $z_{{\bf o}}$, respectively. We define
\begin{align*}
    \eta_e := \gamma_{x,z_{{\bf i}}} \oplus \pi_e \oplus  \gamma_{z_{{\bf o}},y},
\end{align*}
where we recall that, for two paths $\sigma^1$ and $\sigma^2$ with ${\bf e}(\sigma^1) = {\bf s}(\sigma^2)$, the operator $\sigma^1 \oplus \sigma^2$ denotes their concatenation. Therefore, $\eta_e \setminus \gamma \subset \pi_e$ is open. Together with  $\gamma \in \sO_*(x,y)$, we get $\eta_e \in \sO_*(x,y)$, thanks to Remark \ref{rem:omkx}. Notice that $|\eta_e \setminus \gamma| \leq |\pi_e| \leq C_* R_e$. Moreover, since $\gamma^{ {\bf i},e}$ and $ \gamma^{ {\bf o},e}$ are the first and last crossing paths of $\rmA_{R_e}(e)$, one has $\gamma_{x,z_{\bf i}} \cap \Lambda_{R_e}(e) = \varnothing$ and $\gamma_{z_{\bf o},y} \cap \Lambda_{R_e}(e) =  \varnothing$. In addition, $\pi_e \cap \Lambda_{R_e}(e) = \varnothing$ since $\pi_e \subset \rmA_{R_e}(e)$. Hence, $\eta_e \cap \Lambda_{R_e}(e) \subset \{\gamma_{x,z_{{\bf i}}} \cup \pi_e \cup  \gamma_{z_{{\bf o}},y}\} \cap \Lambda_{R_e}(e) = \varnothing$ and so we obtain the claimed result.
\end{proof}
For any path $\gamma$, we denote by $\clo(\gamma)$ the set of all closed edges of $\gamma$. The following result will help us gradually erase closed edges in a geodesic, turning it into an open path with a controllable cost. For more details, see in the proof of Proposition \ref{lemdis}.
\begin{proposition} \label{Lem: application of er 2}
Let $x,y \in \kC_\infty$. Let $\gamma \in \sO_*(x,y)$, and let $e$ be an edge in $\clo(\gamma)$ such that $\{x,y \} \not \subset \Lambda_{3R_e}(e)$. Then there exists another path $\eta_e \in \sO_*(x,y)$ such that: 
         \begin{itemize}
             \item [(i)] $\clo(\eta_e) \cap \Lambda_{R_e}(e) = \varnothing$, and  $\eta_e \setminus \gamma \subset \Lambda_{4R_e}(e)$, and $\eta_e \setminus \gamma$ is open;
             \item [(ii)]  $|\eta_e \setminus \gamma| \leq 2 C_* R_e$.
         \end{itemize}
\end{proposition}
\begin{proof}
Let $\gamma \in \sO_*(x,y)$ with $x,y \in \kC_\infty$ and let $e \in \clo(\gamma)$  satisfying $\{x,y\} \not \subset \Lambda_{3R_e}(e)$.  We consider two cases:

\noindent \textbf{Case 1:} $x,y \notin \Lambda_{3R_{e}}(e)$. Then, by Proposition \ref{Lem: application of er}, there exists another path $\eta_e \in \sO_*(x,y)$, such that the path $\eta_e \setminus \gamma$ is open, $\eta_e \cap \Lambda_{R_e}(e) = \varnothing$, and $|\eta_e \setminus \gamma| \leq C_* R_e$. Hence, we get the desired result.

\noindent \textbf{Case 2:}  there is only $x$ or $y$ in $\Lambda_{3R_{e}}(e)$. By the symmetry, we suppose that $x \in \Lambda_{3R_{e}}(e)$ and $y  \notin \Lambda_{3R_{e}}(e)$. The path $\gamma$ crosses the annulus $\rmA_{R_{e}}(e)$ at least once. We call the last crossing path by $\gamma^{ {\bf o},e}$. Using $\gamma \in \sO_{*}$ and Remark \ref{rem:omkx}, we get $ \gamma^{ {\bf o},e}\in \sC(\rmA_{R_e}(e)) \cap \sO_{*}(\Lambda_{C_*R_e}(e)) $. As $x \in \kC_\infty$, there exists an open path $\xi_{x,\infty}$ joining $x$ to $\infty$.\\
 \underline{Case 2a}:  $x \in \Lambda_{R_e}(e)$. In this case, the open path $\xi_{x,\infty}$ crosses the annulus $\rmA_{R_{e}}(e)$ at least once. Let $\gamma^{ {\bf i},e} \subset  \xi_{x,\infty}$ be the first crossing path of $\aAr$, so $\gamma^{ {\bf i},e} \in \sC(\rmA_{R_e}(e)) \cap \sO_{*}(\Lambda_{C_*R_e}(e))$.
 Since the event $\cV^1_{R_e}(e)$ occurs, there exists an open path $\tilde{\eta}_e$, a geodesic of ${\rm D}^{\rmA_{R_e}(e)}(\gamma^{ {\bf i},e},\gamma^{ {\bf o},e})$, satisfying $|\tilde{\eta}_e| \leq C_* R_e$. Suppose that $\tilde{\eta}_e$ intersects with $\gamma^{ {\bf i},e}$ and $\gamma^{ {\bf o},e}$ at $z_{\bf i}$ and $z_{\bf o}$, respectively  (see Figure \ref{de1} for illustration). By the definition of $\gamma^{ {\bf i},e}$ and $z_{\bf i}$, we have that $\xi_{x,z_{\bf i}}$, the subpath of $\xi_{x, \infty}$  from $x$ to $z_{\bf i}$, is  open  and satisfies $\xi_{x,z_{\bf i}} \subset \Lambda_{3R_e}(e)$. Thus, $\rmD^{\Lambda_{3R_e}(e)}(x,z_{\bf i})< \infty$. Thanks to the definition of $\mathcal{V}^2_{R_e}(e)$, $\rmD^{\Lambda_{4R_e}(e)}(x,z_{\bf i}) \leq C_* R_e$. Let us denote by $\tilde{\xi}_{x,z_{\bf i}}$  the geodesic of $\rmD^{\Lambda_{4R_e}(e)}(x,z_{\bf i})$ and define 
 \be{
 \eta_e :=  \tilde{\xi}_{x,z_{\bf i}} \oplus \tilde{\eta}_e  \oplus \gamma_{z_{\bf o},y}.
 }
 \begin{figure}[htbp]
\begin{center}
\includegraphics[width=12cm]{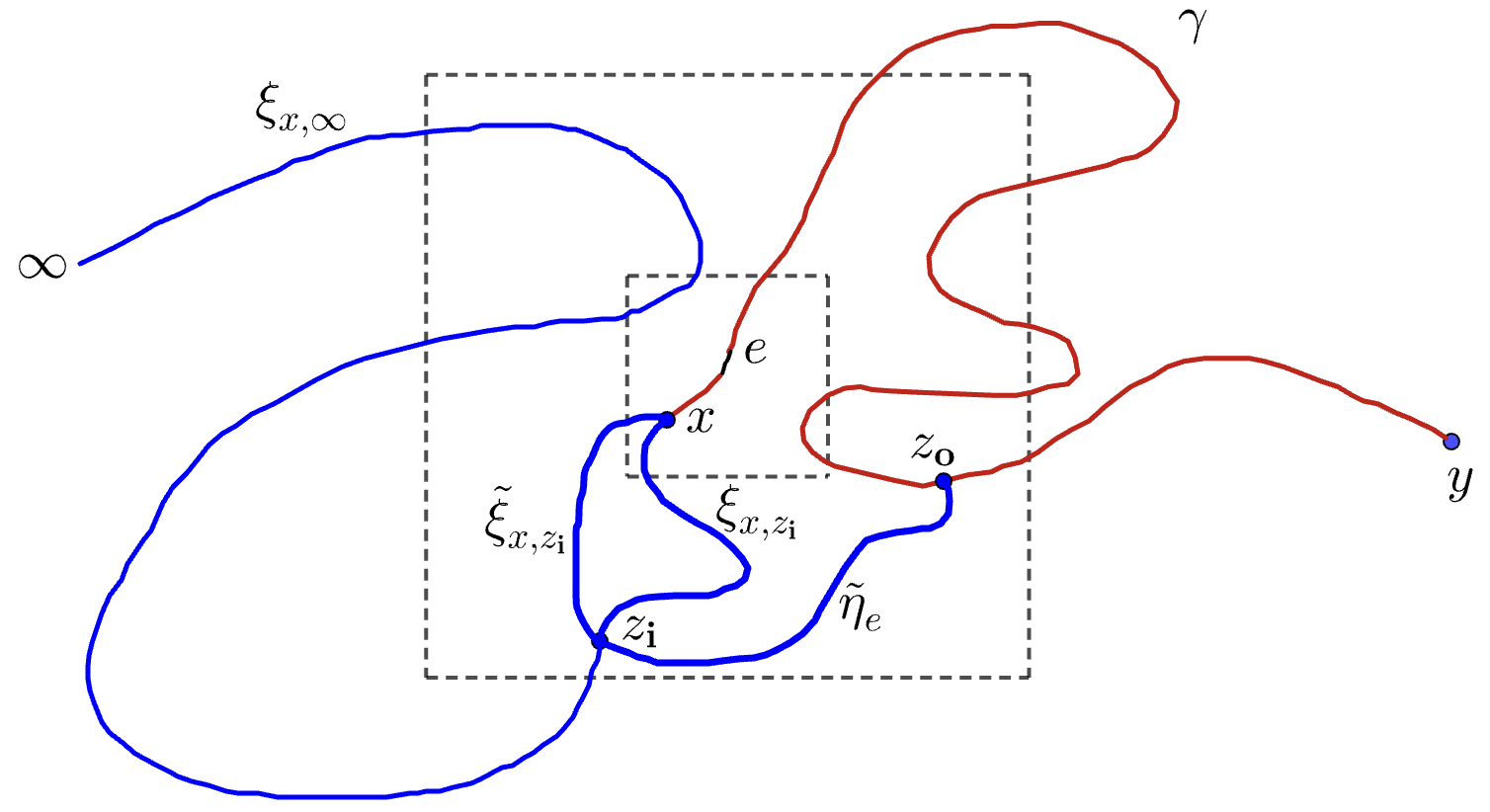}
\caption{Construction of the detour $\eta_e$ avoiding the closed edge $e$, obtained by concatenating the paths $\tilde{\xi}_{x,z_{\bf i}}$, $\tilde{\eta}_e$, and $\gamma_{z_{\bf o},y}$  when $x \in \Lambda_{R_e}(e)$.}
\label{de1}
\end{center}
\end{figure}
We observe that the subpath $\eta_e \setminus \gamma \subset \{ \tilde{\xi}_{x,z_{\bf i}} \oplus \tilde{\eta}_e \}$ consists of only open edges, so $\eta_e \in \sO_*(x,y)$. Moreover, 
\be{
|\eta_e \setminus \gamma| \leq |\tilde{\xi}_{x,z_{\bf i}}|+ |\tilde{\eta}_{e}| \leq 2C_* R_e, \quad \clo(\eta_e) \cap  \Lambda_{R_e}(e)=\varnothing,   
}
 since $\{\eta_e  \cap \Lambda_{R_e}(e)\} \subset \tilde{\xi}_{x,z_{\bf i}}$ is open. Hence, the result follows.
    
 \noindent \underline{Case 2b}: $x \in \rmA_{R_{e}}(e) $. Let $\gamma_{x,e}:= \gamma_{x,x_e}$ be the subpath of $\gamma$ joining $x$ to $e$. Let $\tilde{x}$ be the first point where $\xi_{x,\infty}$ touches $\partial \Lambda_{3R_e}(e)$, and so $\xi_{\tilde{x},x} \subset \Lambda_{3R_e}(e)$. Then, $\eta_{\tilde{x},e}:= \xi_{\tilde{x},x} \oplus \gamma_{x,e}$ crosses the annulus $\rmA_{R_{e}}(e)$ at least once as $\xi_{\tilde{x},x} \cap \partial \Lambda_{3R_e}(e) \neq \varnothing$ and $\gamma_{x,e} \cap \partial \Lambda_{R_e+1}(e) \neq \varnothing$. Notice that since $ \gamma \in \sO_{*}$, we have $ \gamma_{e,x} \in \sO_{*}$, and so $\eta_{\tilde{x},e} \in \sO_{*}$. Let $\gamma^{\textbf{i},e} \subset \eta_{\tilde{x},e}$ be the first crossing path of $\aAr$. Thanks to Remark \ref{rem:omkx} again, $\gamma^{\textbf{i},e} \in \sC(\rmA_{R_e}(e)) \cap \sO_{*}(\Lambda_{C_*R_e}(e))$.
 By the definition of  $\mathcal{V}^1_{R_e}(e)$, 
 there exists a geodesic of ${\rm D}^{\rmA_{R_e}(e)}(\gamma^{{\bf i},e},\gamma^{{\bf o},e})$ inside $\rmA_{R_e}(e)$, denoted by $\tilde{\eta}_{e}$, that includes only open edges and satisfies $|\tilde{\eta}_{e}| \leq C_* R_{e}$. Suppose that $\tilde{\eta}_e$ intersects with $\gamma^{\textbf{i},e}$ and  $\gamma^{\textbf{o},e}$  at   $z_{\bf i}$  and $z_{\bf o}$, respectively.
 
\noindent If $z_{\bf i} \in \tilde{\eta}_e \cap \gamma_{x,e} \cap \eta_{\tilde{x},e}$ (see Figure \ref{defcase2b}-A for illustration),  we set  
$$
\eta_e :=  \gamma_{x,z_{\bf i}} \oplus \tilde{\eta}_e  \oplus \gamma_{z_{\bf o},y}.
$$
\begin{figure}[htbp]
    \centering
    \subfloat[\centering if $z_{\bf i} \in \tilde{\eta}_e \cap \gamma_{x,e}$, then $\eta_e$ is obtained by concatenating $\gamma_{x,z_{\bf i}}$, $\tilde{\eta}_e$, and  $\gamma_{z_{\bf o},y}$.]{{\includegraphics[width=8.20cm]{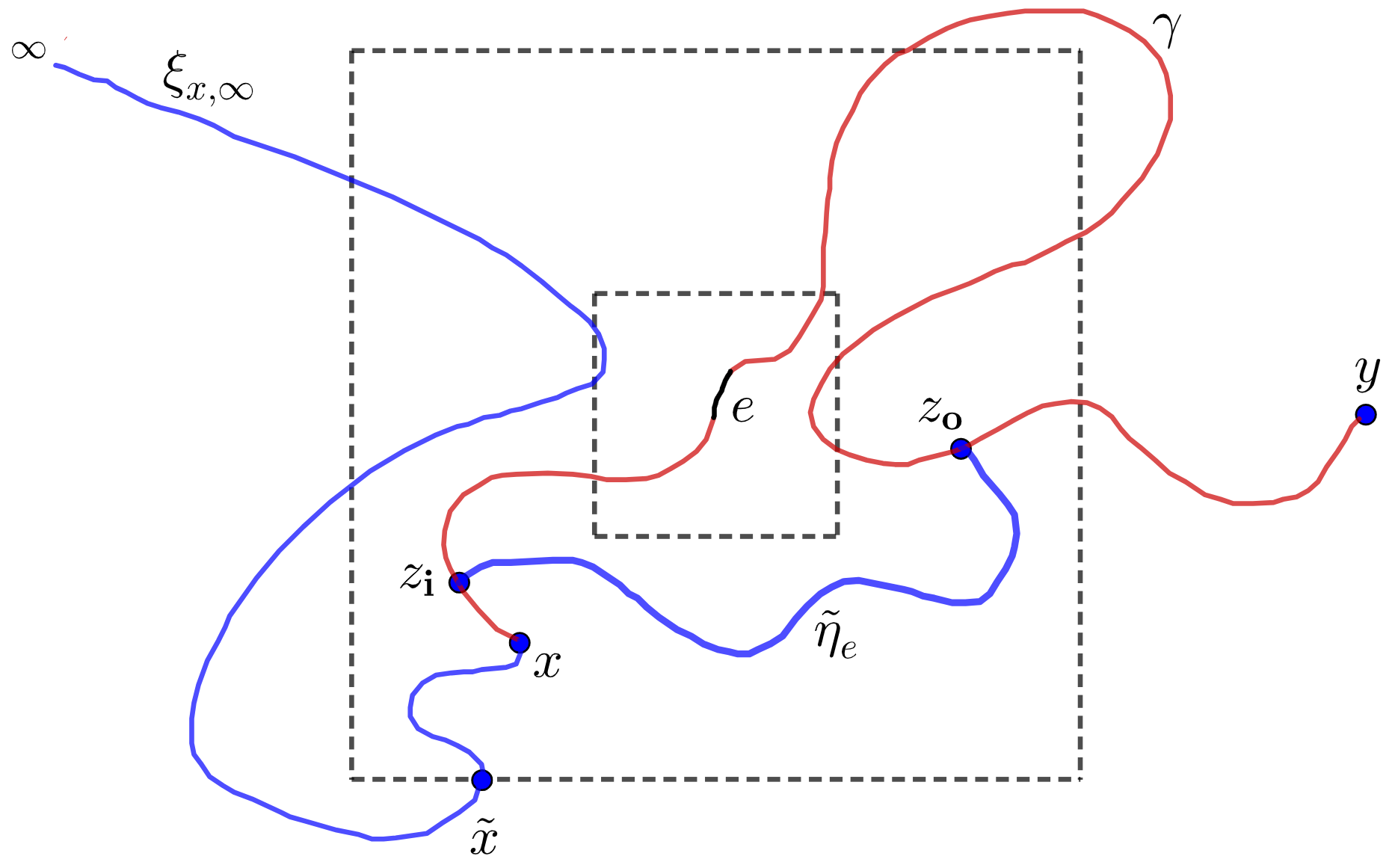} }}%
    \,
    \subfloat[\centering if $z_{\bf i} \in \tilde{\eta}_e \cap \xi_{\tilde{x},x}$, then $\eta_e$ is obtained by concatenating $ \tilde{\xi}_{x,z_{\bf i}},  \tilde{\eta}_e$, and $\gamma_{z_{\bf o},y}$.]{{\includegraphics[width=8.4cm]{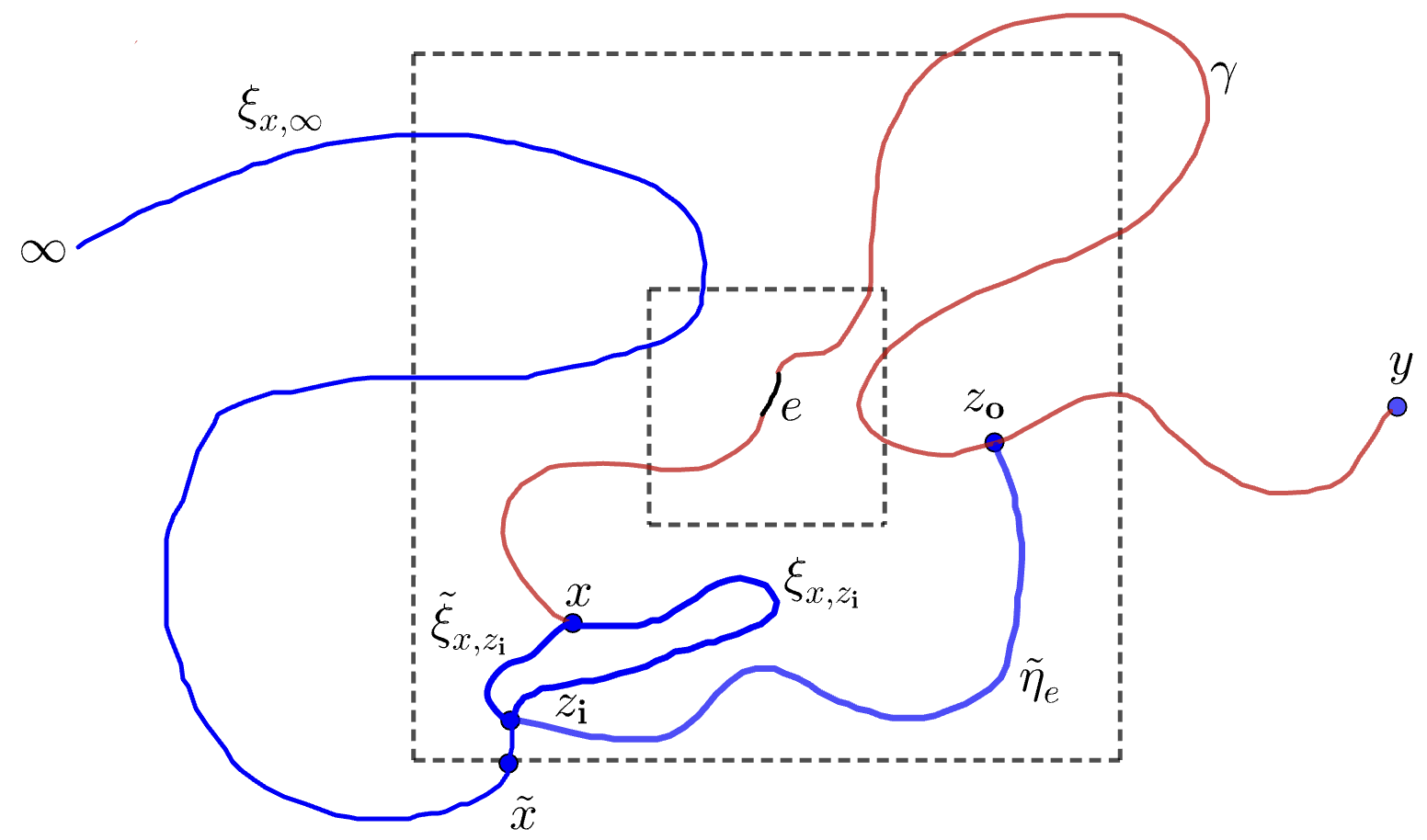} }}%
    \caption{Construction of the detour $\eta_e$ avoiding the closed edge $e$ when $x \in \rmA_{R_{e}}(e) $.}%
    \label{defcase2b}
\end{figure}
Therefore, the path $\eta_e \setminus \gamma \subset \tilde{\eta}_e $ is open, so $\eta_e \in \sO_*(x,y)$. In addition,  
    $$
     |\eta_e \setminus \gamma| \leq  |\tilde{\eta}_{e}| \leq C_* R_e, \quad \eta_e  \cap \Lambda_{R_e}(e)= \varnothing.$$ 
\noindent  If  $z_{\bf i} \in \tilde{\eta}_e \cap \xi_{\tilde{x},x} \cap \eta_{\tilde{x},e}$ (see Figure \ref{defcase2b}-B for illustration), then $\rmD^{\Lambda_{3R_e}(e)}(x,z_{\bf i})< \infty$. By the definition of $\mathcal{V}^2_{R_e}(e)$, we get that $\rmD^{\Lambda_{4R_e}(e)}(x,z_{\bf i}) \leq C_* R_e$. We take $\tilde{\xi}_{x,z_{\bf i}}$  a geodesic of $\rmD^{\Lambda_{4R_e}(e)}(x,z_{\bf i})$ and set 
 \be{
 \eta_e:=  \tilde{\xi}_{x,z_{\bf i}} \oplus \tilde{\eta}_e  \oplus \gamma_{z_{\bf o},y}.
 } 
Hence, 
$\eta_e \setminus \gamma \subset \{\tilde{\xi}_{x,z_{\bf i}} \oplus \tilde{\eta}_e \}$ consists of only open edges, so $\eta_e \in \sO_*(x,y)$. Also,  
    $$
     |\eta_e \setminus \gamma| \leq |\tilde{\xi}_{x,z_{\bf i}}|+ |\tilde{\eta}_{e}| \leq 2C_* R_e, \quad \clo(\eta_e) \cap \Lambda_{R_e}(e)= \varnothing,$$  
as $ \{\eta_e  \cap \Lambda_{R_e}(e)\} \subset \tilde{\xi}_{x,z_{\bf i}}$ is open. We complete the proof for this case.
\end{proof}
Next, we provide an estimate for the impact of resampling an edge on the first passage time via a truncation of effective radius.
\begin{proposition}\label{prop1} Let $(\hat{R}_e)_{e \in \cE}$ be the sequence of truncated effective radii defined as
\begin{align}\label{def: truncated radius}
\forall \, e \in \cE: \quad 
\hat{R}_e:=\min\{C_*R_e, \, \log^2n\}.
\end{align}
 Then for any $x, y \in \Z^d$ and $e \in \cE$,
    \begin{align} \label{i}
       0  \leq (\et_{x,y}(\log^2 n,t_{e^c}) - \et_{x,y}(1,t_{e^c}))\I(t_e =1)
        \leq (\log^2 n \I(\cU_{e})+  \hat{R}_e) \I(e \in \ega),
    \end{align}
    where $\rmT_{x,y}(t_e,t_{e^c})$ denotes $\rmT(x,y)$ as a function of weights $t_e$ and $(t_{e'})_{e' \neq e}$, and
    $\gamma$ is a geodesic of $\rmT_{x,y}(t_e,t_{e^c})$, and
    \begin{align*}
    \cU_{e} := \{ 3R_e \geq  \min \{    \|e-x\|_{\infty},   \|e-y\|_{\infty} \}\}.
    \end{align*}
\end{proposition}
\begin{proof}
 Since $\rmT_{x,y}$ is increasing in $t_e$, the first inequality in \eqref{i} is obvious. Next, we consider the second inequality. Assume that $e$ is  open (i.e. $t_e =1$).  If  $e \notin  \gamma$, then closing $e$ ($t_e = \log^2n$) has no effect on the geodesic, and hence
\begin{align*}
    \et_{x,y}(\log^2 n, t_{e^c}) - \et_{x,y}(1, t_{e^c}) = 0.
\end{align*}
If $e\in\gamma$, it is clear that
\begin{align*}
    \et_{x,y}(\log^2 n, t_{e^c}) - \et_{x,y}(1, t_{e^c}) \leq \log^2 n.
\end{align*}
If $ e \in  \gamma$ and $\cU_{e}^c$ occurs, then neither $x$ nor $y$ belongs to $\Lambda_{3R_e}(e)$. Applying Proposition \ref{Lem: application of er} to $\gamma \in \sO_{*}(x,y)$ and $e \in \gamma$, there exists another path $\eta_e$ between $x$ and $y$ such that
\begin{align*}
    \et_{x,y}(\log^2 n, t_{e^c}) - \et_{x,y}(1, t_{e^c})\leq \rmT(\eta_e \setminus \gamma) = |\eta_e \setminus \gamma| \leq  C_* R_e. 
\end{align*}
Combining the last three  estimates, we get the desired result. 
\end{proof}
\subsection{Proof of Proposition \ref{Lem: effective radius}}\label{sec 2.2}
We first review some intrinsic properties of chemical distance and crossing cluster of supercritical Bernoulli percolation. We recall that  an open cluster $\kC$ is a crossing cluster of a box $\Lambda$ if for all $d$ directions, there is an open path in $\kC \cap \Lambda$ connecting the two opposite faces of $\Lambda$.
\begin{lemma}\cite[Theorem 2]{pisztora1996surface} \label{hole}
There exists a constant $c >0$ such that for all $x \in \Z^d$ and $t >0$
 \begin{align}\label{Claim: 2e0}
      \pr (\|x- x^*\|_{\infty} \geq t) \leq c^{-1} \exp (-c t^{d-1}). 
 \end{align}
 \end{lemma}
\begin{lemma} \cite[Theorem 7.68]{grimmett1999percolation}\label{Lem: pro of crossing event}
There exists a constant $c >0$ such that for all $t>0$,
\begin{align*}
 \pr( \Lambda_{t} \text{ has an } \text{open crossing cluster}) \geq 1-c^{-1} \exp(-c t).
\end{align*}
\end{lemma}
\begin{lemma} 
\cite[Corollary 2.2, Lemma 2.3]{garet2009moderate}
 \label{Lem: large deviation of graph distance Dlambda}
 There exist constants $\rho \geq 1 $ and $ c >0$ such that for every $x \in \Z^d$ and $t \geq \rho \|x\|_{\infty} $,
\begin{align}\label{large deviation of D,0-x}
 \max \{ \pr(\infty >\rmD(0,x) \geq  t ),\pr(\rmD^* (0,x) \geq  t) \} \leq c^{-1} \exp(-c t).
\end{align}
\end{lemma}
\begin{lemma} \cite[Lemma 7.104]{grimmett1999percolation} \label{lem: twodisjointclusters}
    For any $\varepsilon > 0$, there exists $c>0$ such that for all $t >0$,
    \begin{align*}
        \pr(\, \exists \, \text{two disjoint open clusters with diameter at least $\varepsilon t$ in $\Lambda_t$} ) \leq c^{-1} \exp(-c t).
    \end{align*}
\end{lemma}
Given $t > \ell > 0$, we define slab
\begin{align*}
 S(\ell,t):= \{ x=(x_1,\ldots,x_d) \in \Lambda_t: 0 \leq x_1 \leq \ell\}.
\end{align*}
\begin{lemma} \label{Lem: path co boundary in slab}  There exists a constant $C>0$
 such that for all $ C \log t \leq \ell < t$,
\begin{align}
\pr (\exists \, \gamma \in \bO(S(\ell,t)): \Diam(\gamma) \geq \ell, \gamma \not \xleftrightarrow{S(\ell,t)} \partial \Lambda_{t} ) \leq C \exp(-\ell/C).
\end{align}
\end{lemma}
\begin{proof} We say an open cluster $\kC$ crosses the slab $S(\ell,t)$ if for all $d$ directions, there is an open path of $\kC \cap S(\ell,t)$ connecting two opposite faces of $S(\ell,t)$. Observe that $(d-1)$ pairs of opposite faces of $S(\ell,t)$ are contained in $\partial \Lambda_t$ (except two opposite faces $\{x \in S(\ell,t): x_1 = 0\}$ and $\{x \in S(\ell,t): x_1 = \ell \}$). Thus, if  $\kC$ is a crossing cluster in $S(\ell,t)$, then $\kC \cap \partial \Lambda_t \neq \varnothing$. Let $\cB_m(X)$ be the set of all boxes of side-length $m$ in $X \subset \Z^d$. Define
\begin{align*}
    \cW^1_{\ell,t}:= \{ \text{$\exists$ crossing cluster $\kC$ in $S(\ell,t)$ that contains a crossing cluster of  $\Lambda$ for all $\Lambda \in \cB_{\ell}(S(\ell,t))$}\}.
\end{align*}
It follows from Lemmas \ref{Lem: pro of crossing event} and \ref{lem: twodisjointclusters} that there exists a constant $C_1 > 0$ such that  for all $\ell \geq C_1 \log t$, 
$$\pr(\cW^1_{\ell,t}) \geq 1 - C_1 \exp (-\ell/C_1).$$ Thanks to union bound and Lemma \ref{lem: twodisjointclusters} again, there exists a constant $C_2 > 0$ such that for all $\ell \geq C_2 \log t$,
\begin{align*}
    \pr(\cW^2_{\ell,t}) \geq 1- C_2 \exp (-\ell/C_2), 
\end{align*}
where 
\begin{align*}
    \cW^2_{\ell,t}: = \{\forall \, \Lambda \in \cB_{\ell}(S(\ell,t)):\,  \text{there is at most one open clusters with diameter at least $\ell/2$ in $\Lambda$} \}.
\end{align*}
Suppose that $\cW^1_{\ell,t} \cap \cW^2_{\ell,t}$ occurs. Then, there exists a crossing cluster $\kC$ in $S(\ell,t)$ such that for all $\Lambda \in \cB_{\ell}(S(\ell,t))$ and any open path $\eta \subset \Lambda$ with diameter at least $\ell/2$, $\kC \cap \eta \neq \varnothing$. Notice that if $\gamma \in  \bO(S(\ell,t))$ with $\Diam(\gamma) \geq \ell$, there exist $x,y \in \gamma$ such that $\|x-y\|_\infty \geq \ell$. Thus, there exists a subpath $\gamma' \subset \gamma_{x,y} \cap \Lambda_{\ell/2}(x)$ such that $\Diam(\gamma')\geq \ell/2$. Moreover, there exists a box $\Lambda\in \kB_{\ell}(S(\ell,t))$ such that $\Lambda$ contains $\Lambda_{\ell/2}(x)\cap S(\ell,t)$. Therefore, $\gamma' \cap \kC \neq \varnothing$ due to the assumption of  $\cW^1_{\ell,t} \cap \cW^2_{\ell,t}$. Thus $\gamma \xleftrightarrow{S(\ell,t)} \partial \Lambda_{t}$ since $\kC \cap \partial \Lambda_t \neq \varnothing$. In conclusion, 
\begin{align*}
    \pr (\forall \gamma \in \bO(S(\ell,t)): \Diam(\gamma) \geq \ell, \gamma  \xleftrightarrow{S(\ell,t)} \partial \Lambda_{t} ) \geq \pr(\cW^1_{\ell,t} \cap \cW^2_{\ell,t}) \geq 1- C\exp(-\ell/C),
\end{align*}
with $C:= 2\max\{C_1,C_2\} > 0$.
\end{proof}

\subsubsection{Good box} 

We present in this part a condition called \textit{good box} under which the bypass in the definition of effective radius can be easily constructed. Let $\rho$ be the constant defined as in Lemma \ref{Lem: large deviation of graph distance Dlambda}, and set
 \begin{align*}
  N_{\rho}:=\lfloor  \tfrac{N}{ 16 \rho^2} \rfloor.
 \end{align*}
 Roughly speaking, a good box possesses the large-scale geometry of its percolation cluster so similar to  Euclidean space that guarantees the feasibility of constructing the effective radius.
\begin{definition}\label{Def: good annulus} 
For each $e \in \cE$, we say that the box $ \Lambda_{3N}(e)$ is  \textbf{good} if the following hold:
\begin{itemize}
      \item [(i)] there exists an open  cluster  $\kC$ in $\Lambda_{3N}(e)$ that contains a crossing cluster of  $\Lambda$ for all $\Lambda \in \cB_{N_\rho}(\Lambda_{3N}(e))$;
    \item [(ii)] for all $x,y \in \Lambda_{3N}(e)$  with $\|x-y\|_\infty \leq 2 N_{\rho}$, if $ \rmD(x,y)< \infty$, then  $\rmD^{\Lambda_{4N}(e)}(x,y)=\rmD(x,y) \leq 4\rho N_{\rho}$;
    \item  [(iii)] for all $x,y \in \aAn$ with ${\rm d}_{\infty} (\{x,y\},\partial \aAn) \geq N/2$ and $\|x-y\|_{\infty} \leq 2N_\rho$, if $\rmD(x,y) < \infty$, then $\rmD^{\aAn}(x,y) = \rmD(x,y) \leq 4\rho N_\rho$;
    \item [(iv)] if $\pi \in  \kP(\Lambda_{3N}(e)) \cap \sO_{*}(\Lambda_{C_*N}(e))$  satisfies  $\Diam(\pi) \geq N_\rho$, then  $\pi \cap \kC \neq \varnothing$.
\end{itemize}
\end{definition}
\begin{lemma} \label{Lem: being bad annulus}
     There exists a constant $ c >0$ such that for all $e \in \cE$ and $1 \leq N \leq \exp(c\log^2 n)$,
    \begin{align*}
         \pr(\Lambda_{3N}(e) \text{ is } \textbf{good} ) \geq 1- c^{-1}\exp(-c N).
    \end{align*}
\end{lemma}
 Among the properties of a good box, (iv) is the most complicated to prove and is not standard; it is treated in the following result.
\begin{lemma} \label{lem: crossing cluster}
 For any  $p> p_c(d) $, there exists a constant $c_* \in (0,1)$ such that  for all $t \geq 1/c_*$ and $1 \leq  N \leq \exp( c_* \log^2 n)$,
\begin{align}\label{Eq: exponetial tail of en}
\pr(\kE_N) \leq c_*^{-1} \exp(-c_* N),
\end{align}
 where
\begin{align*}
   \kE_N:= \{  \exists \,  \textrm{crossing cluster }  \kC \subset \Lambda_{2N},\,\exists \,  \pi \in \kP(\Lambda_N)  \cap  \sO_{*}(\Lambda_{tN}),  & \Diam(\pi)\geq  N/2: \, \pi \cap \kC = \varnothing \}.
\end{align*}
\end{lemma}
In the event $\kE_N$, if $\pi \in \kP(\Lambda_N) \cap \bO(\Lambda_N)$, i.e., $\pi$ was opened, then \eqref{Eq: exponetial tail of en} obviously follows from Lemma \ref{lem: twodisjointclusters}. To prove this lemma, the key point here is to show that $\pi$ does not contain many closed edges, and thus has a large number of open segments. The remainder of the proof follows by using the slab technique.
\begin{proof}
By scaling $c_*$, it is sufficient to prove the result for $N$ sufficiently large. Hence, in the following computations, we assume that $N$ is large enough. In the case $d=2$, by the standard arguments using the planarity of $\Z^2$, we can prove the claim even for general paths $\pi$. Now we assume that  $d\geq 3$. Let us define
    \begin{align*}
        \cL_N := \{ \forall \text{ crossing cluster } \kC \subset \Lambda_{2N},\, \forall \text{ open cluster }\kD \subset \Lambda_{2N}  \text{ with }\Diam (\kD) \geq N: \kC \cap \kD \neq \varnothing\}.
    \end{align*}
Then we have 
\ben{ \label{elep}
(\kE'_N)^c \cap \cL_N \subset  (\kE_N)^c,
}
where 
\begin{align*} 
  \kE'_N := \{ \exists \, \pi \in \kP(\Lambda_N)  \cap  \sO_{*}(\Lambda_{tN}), \Diam(\pi)\geq  N/2: \pi \not \xleftrightarrow{\Lambda_{2N}}  \partial \Lambda_{2N}  \}.
\end{align*}
By Lemma \ref{lem: twodisjointclusters}, there exists $c_1>0$ such that $\pp(\cL_N^c) \leq c_1^{-1}\exp(-c_1N)$, and thus \eqref{elep} gives
\begin{align}\label{E'n}
    \pr(\kE_N) \leq \pr(\kE'_N)+ c_1^{-1}\exp(-c_1 N).
\end{align}
To estimate $\pr(\kE'_N)$, we first show that the geodesics do not have many closed edges. More precisely,  for any  $a>0$ there  exists $\varepsilon \in(0,1/3)$ such that  for all $ N \leq \exp(\varepsilon \log^2 n)$ and $t \geq 1/\varepsilon$,
 \begin{align}\label{Event: control of closed edges}
 \pp((\cL_N')^c) \leq \varepsilon^{-1}\exp(-\varepsilon N), 
  \end{align}
where
\begin{align*}
\cL_N':=\{\forall \, \xi =(\xi_e)_{e \in \cE} \text{ s.t. } \xi_e = t_e \forall \, e \in \Lambda_{tN}, \forall \,  \gamma \in \cup_{x,y \in \Lambda_N}\sO^\xi(x,y): |\clo_\xi(\gamma)| \leq \tfrac{aN}{\log N}\}.
  \end{align*}
Here $\clo_\xi(\gamma)=\{e \in \gamma: \xi_e =  \log^2n\}$ and $\sO^\xi(x,y)$ is the set of all geodesics from $x$ to $y$ inside $\Z^d$ with respect to the configuration $\xi$. Indeed, by Lemmas \ref{hole} and \ref{Lem: large deviation of graph distance Dlambda} with the remark that $d\geq 3$, there exists $ \alpha \in (0,1)$ such that 
 \begin{align}\label{Eq: A_N}
     \pr(\cA_N) \geq 1 -\alpha^{-1} \exp(-\alpha N), 
 \end{align}
 where 
 \begin{align*}
     \cA_N:= \{ \forall x \in \Lambda_N: \|x -x^*\|_1 \leq \sqrt{N}/2\} \cap  \{ \forall x,y \in \Lambda_{2N} \cap \kC_{\infty}: \rmD(x,y) \leq \alpha^{-1} N\}.
 \end{align*}
On $\cA_N$, for all $t \geq 2/\alpha$ and $x,y \in \Lambda_N$,
\begin{align*}
    \rmT^{\Lambda_{tN}}(x,y)  &\leq \rmT^{\Lambda_{tN}}(x,x^*)+ \rmT^{\Lambda_{tN}}(x^*,y^*)+ \rmT^{\Lambda_{tN}}(y^*,y)\\
    & \leq \log^2 n(\|x-x^*\|_1+\|y-y^*\|_1)+ \rmD(x^*,y^*)  \leq \sqrt{N} \log^2 n+ \alpha^{-1} N.
\end{align*}
Therefore, for all $\xi =(\xi_e)_{e \in \cE}$ such that $\xi_e = t_e \, \forall e \in \Lambda_{tN}$ and  $ \gamma \in \cup_{x,y \in \Lambda_N}\sO^\xi(x,y)$,
\be{
|\clo_\xi(\gamma)| \leq \frac{\rmT_\xi(x,y) }{\log^2 n} \leq \frac{\rmT_\xi^{\Lambda_{tN}}(x,y) }{\log^2 n} = \frac{\rmT^{\Lambda_{tN}}(x,y) }{\log^2 n} \leq \sqrt{N} + \frac{\alpha^{-1}N}{ \log^2n} \leq a N/\log N,
}
for all $1/\varepsilon \leq N \leq \exp(\varepsilon \log^2 n)$ with  $\varepsilon=\varepsilon(a,\alpha)$ sufficiently small. Hence, \eqref{Event: control of closed edges} follows. \\

Suppose now that $\kE'_N$ occurs, and let $\pi \in \kP(\Lambda_N)  \cap  \sO_{*}(\Lambda_{tN})$ be a path that realizes $\kE'_N$. We assume further that the diameter of $\pi$ is achieved in the first coordinate direction, i.e., there exist $u,v \in \pi$ with $u_1 - v_1 = \Diam(\pi) \geq N/2$. Let $\pi^*$ be the subpath of $\pi$ from $u$ to $v$.
Since $\pi \in \kP(\Lambda_N) \cap \sO_{*}(\Lambda_{tN})$, so is $\pi^*$. The main idea of the proof can be summarized as follows. We divide the box $\Lambda_{2N}$ into slabs of size $L \asymp \log N$ in the first coordinate. Observe that $\pi^*$ has to cross at least $\lfloor N/2L \rfloor$ slabs (since $u_1-v_1 \geq N/2$).  By \eqref{Event: control of closed edges}, we can assume that $\pi^*$ has at most $a N/\log N$ closed edges for some constant $a>0$. Consequently, choosing $L \leq \log N /(5a)$, it follows that $\pi^*$ contains at least $\lfloor N/(4L) \rfloor$ open segments inside disjoint slabs, each of diameter at least $L$. For each such segment, the probability that it is not connected to the boundary of $\Lambda_{2N}$ inside its slab is at most $\exp(-cL)$ for some constant $c>0$. Since the slabs are disjoint, these connecting events are independent. Hence, the probability that  one of these segments fails to connect to the boundary  is at most $\exp(-cN/4)$.
\begin{figure}[htbp]
\begin{center}
\includegraphics[width=8cm]{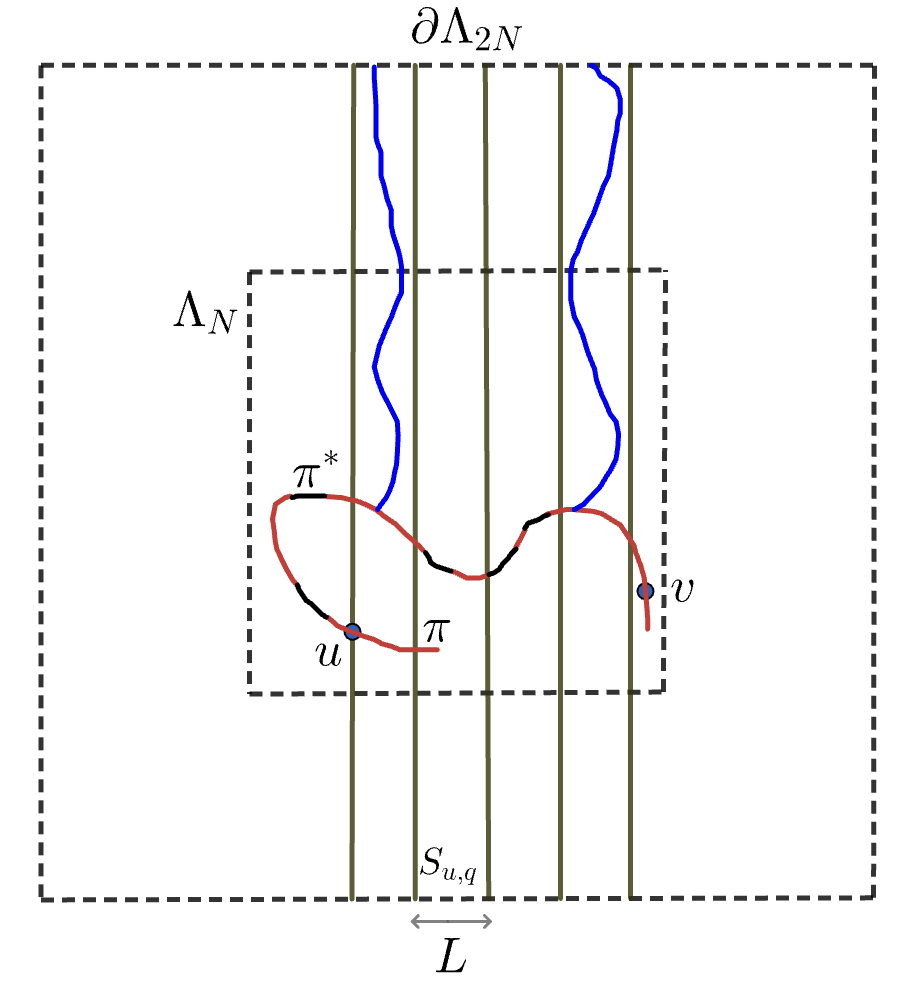}
\caption{An illustration of the estimate of $\mathbb{P}(\mathcal{E}'_N)$.  
The black lines represent closed edges, while the red lines represent open edges.  
The open paths connecting the open segments to $\partial \Lambda_{2N}$ inside their slabs are depicted in blue.
} 
\label{pitob}
\end{center}
\end{figure}

Let $L \asymp \log N$ be  an integer  chosen later and define   $K := \lfloor \tfrac{N}{2L} \rfloor$. For $ 0 \leq q \leq K$, we define surfaces and slabs:
\begin{align*}
    & H_{u,q} := \{x=(x_1,\ldots,x_d) \in \Lambda_{2N}: x_1 = u_1 +qL\}; \\
  & S_{u,q} := \{x=(x_1,\ldots,x_d) \in \Lambda_{2N}:  u_1 +(q-1) L \leq x_1 < u_1 + qL \},
\end{align*}
where recall that $u$ is the starting point of $\pi^*$. 
Notice that for all $q \in  [K] $, $\pi^*$ has to cross the slab $S_{u,q}$ from $H_{u,q-1} $ to $H_{u,q}$.  For each $q \in [K]$, let  $\pi^{*}_q \subset \pi^*$ be a  subpath of $\pi^* \cap S_{u,q}$ such that $\Diam(\pi^*_q) \geq L$. Hence, on the event $\kE'_N$,  
$$
\forall q \in  [K]: 
     \quad \pi^*_q \subset S_{u,q}, \,\, \Diam(\pi^*_q) \geq L, \,\,  \pi^*_q \not \xleftrightarrow{S_{u,q} } \partial \Lambda_{2N}.
$$
Let  $I :=  \{q \in [K]: \pi_q^* \in \bO(\Lambda_N) \}$ be the set of open subpaths of $\pi^*$. As $\pi^* \in  \kP(\Lambda_N) \cap \sO_{*}(u,v;\Lambda_{tN})$, there exist a configuration $\xi=(\xi_e)_{e \in \cE}$ with $\xi_e = t_e$ for all $e \in \Lambda_{tN}$ and a geodesic $\eta \in \sO^\xi$ such that $\pi^* \setminus \eta$ is open. Let $u', v'$ be the first and last points of $\eta$ intersecting with $\Lambda_N$, respectively, ordered from ${\bf s}(\eta)$ to ${\bf e}(\eta)$. Let $\eta'$ be the subpath of $\eta$ from $u'$ to $v'$, so that $\eta' \in \cup_{x,y \in \Lambda_N}\sO^\xi(x,y)$. Since $\pi^* \subset \Lambda_N$, it follows that $\pi^* \setminus \eta = \pi^* \setminus \eta'$ is open. Then for any $a>0$, conditioned on $\cL_N'$ (defined in \eqref{Event: control of closed edges} with $t\geq 1/\varepsilon$ and $\varepsilon=\varepsilon(a)$ sufficiently small), we get that
 \begin{align*}
     \clo(\pi^*) \leq  \clo_\xi(\eta') \leq \frac{aN}{\log N},
 \end{align*}
 which implies that
\begin{align} \label{car(I)}
 |I| \geq K - |\clo(\pi^*)| \geq K - \frac{a N}{ \log N} \geq \frac{N}{4L}, 
\end{align}
when $L \leq  \log N/(5a)$. Therefore,
\begin{align*}\label{bound: e' cap }
    &\pr(\kE'_N  \cap \cL'_N )\notag \\
     & \leq \pr \Big( \exists \, u \in \Lambda_N,  \, \,\exists \, I \subset  [K],  |I| \geq  \tfrac{N}{4L}:  \forall q \in I, \, \exists \,  \gamma_{q} \in \bO(\Lambda_N), \gamma_q \subset S_{u,q},  \, \,\Diam(\gamma_q)\geq L , \gamma_q \not \xleftrightarrow{S_{u,q}} \partial \Lambda_{2N}\Big) \notag \\
     & \leq \sum_{u \in \Lambda_N} \, \, \sum_{I \subset [K], \,|I| \geq   \tfrac{N}{4L}} \pr \Big( \forall \, q \in I,
     \exists \,  \gamma_{q}\in \bO(\Lambda_N): \Diam(\gamma_q)\geq L, \gamma_q \subset S_{u,q}, \gamma_q \not \xleftrightarrow{S_{u,q}} \partial \Lambda_{2N} \Big) \notag \\
     & \leq \cO(1) N^d 2^{N/L} \Big( \pr \big(\exists \, \gamma \in  \bO(\Lambda_N), \gamma \subset S_{0,1}: \Diam(\gamma) \geq L, \gamma \not \xleftrightarrow{S_{0,1}} \partial \Lambda_{2N}  \big)  \Big)^{\lfloor\frac{N}{4L}\rfloor},
\end{align*}
where in the last line we have used the independence of events in disjoint slabs. By Lemma \ref{Lem: path co boundary in slab}, there exists a constant $C > 0$ such that if $L \geq C \log N$, 
\begin{align*}
\pr (\exists \, \gamma \in \bO(\Lambda_N) \cap S_{0,1}: \Diam(\gamma) \geq L, \gamma \not \xleftrightarrow{S_{0,1}} \partial \Lambda_{2N} ) \leq C \exp(-L/C).
\end{align*}
Thus, for all $N$ sufficiently large
\ben{ \label{el2p}
\pr(\kE'_N \cap \cL'_N ) \leq \exp(-N/(8C)).
}
Finally, we  clarify the choice of parameters. We take $a=1/(10C)$ and $L= \lceil C \log N \rceil$ under which $L\leq  \log N/(5a)$ and $L \geq C \log N$. Using  \eqref{el2p} and \eqref{Event: control of closed edges},
\be{
\pr(\kE'_N  ) \leq \varepsilon^{-1} \exp(-\varepsilon N) + \exp(-N/(8C)),
}
with $\varepsilon=\varepsilon(a)$ given in \eqref{Event: control of closed edges}. Combining this with \eqref{E'n} and taking $c_*=\frac{1}{4}\min\{\varepsilon,c_1, 1/8C\}$, we obtain the desired result. 
 \end{proof}
\begin{proof}[Proof of Lemma \ref{Lem: being bad annulus}] Recall the definition of good box in \ref{Def: good annulus} with properties (i)--(iii). Fix $e \in \cE$. It follows from Lemmas \ref{Lem: pro of crossing event} and \ref{lem: twodisjointclusters} that
\begin{align}\label{satisying i}
     \pr(\Lambda_{3N}(e) \text{ does not satisfy (i)}) \leq c_1^{-1} \exp(-c_1 N_\rho),
\end{align}
for some constant $c_1 >0$. Next, by Lemma \ref{Lem: large deviation of graph distance Dlambda}, there exists $c_2 >0$ such that
\begin{align}\label{Good box: proper ii}
& \pr(\Lambda_{3N}(e) \text{ does not satisfy (ii)}) \notag \\
& \leq 
    \pr(\exists \,  x,y \in \Lambda_{3N}(e), \|x-y\|_\infty < 2 N_{\rho}, \rmD(x,y)< \infty, \rmD(x,y) \geq  4\rho N_{\rho}) \leq c_2^{-1} \exp(-c_2 N).
\end{align}
To deal with (iii), we observe that if the property (iii) is not satisfied, then 
 \begin{flalign*}
    \textrm{(iii')} &&\exists \,  x,y \in \aAn \text{ such that } {\rm d}_{\infty} (\{x,y\},\partial \aAn) \geq N/2, \|x-y\|_{\infty} \leq 2N_\rho,   \rmD(x,y) \in [4\rho N_\rho,\infty). &&
\end{flalign*}
Thanks to union bound and Lemma \ref{Lem: large deviation of graph distance Dlambda} again, there exists $c_3 >0 $ such that
\begin{align} \label{(ii)}
  \pr(\Lambda_{3N}(e) \text{  does not satisfy (iii)}) & \leq   \pr(\Lambda_{3N}(e) \text{  does not satisfy (iii')}) \notag \\
  & \leq c_3^{-1} |\aAn|^2 \exp(-c_3  N_\rho) \leq c_3^{-1} \exp( -c_3^2 N_\rho).
\end{align}
Suppose now that $\Lambda_{3N}(e)$ satisfies (i) but not (iv). Then it holds that
 \begin{align*}
   \textrm{(iv') } & \text{   there exists }   \pi \in \kP(\Lambda_{3N}(e)) \cap \sO_{*}(\Lambda_{C_*N}(e)) \text{ such that } \Diam(\pi) \geq N_\rho ,\notag\\
   & \quad \text{ and an open cluster } \kC \subset \Lambda_{3N}(e) \text{ such that } \kC \text{ crosses all } \Lambda \in \cB_{N_\rho}(\Lambda_{3N}(e)): \pi \cap \kC = \varnothing.  
 \end{align*}
  Notice that given a path $\pi \subset \Lambda_{3N}(e)$ with $\Diam(\pi) \geq N_\rho$, there exist $x,y \in \pi$ such that $\|x-y\|_\infty \geq N_{\rho}$. That is, there exists a subpath $\pi' \subset \pi_{x,y} \cap \Lambda_{N_{\rho}/4}(x)$ such that $\Diam(\pi')\geq N_{\rho}/8$, where   $\pi_{x,y}$ is the subpath of $\pi$ from $x$ to $y$. Notice that we always can find $\Lambda\in \kB_{N_{\rho}/2}(\Lambda_{3N}(e))$ such that $\Lambda$ contains $\Lambda_{N_{\rho}/4}(x)\cap \Lambda_{3N}(e)$, and hence  $\pi' \subset \Lambda$ with $\Diam(\pi')\geq N_{\rho}/8$. Therefore,  it follows from (iv') that there exists a vertex $z \in \Lambda_{3N}(e)$, a crossing cluster $\kC' \subset \Lambda_{N_\rho/2}(z) \subset  \Lambda_{3N}(e)$, and a subpath $\pi' \in \kP(\Lambda_{N_\rho/4}(z)) \cap \sO_{*}(\Lambda_{C_*N_\rho/4}(z)) $ with $\Diam(\pi')\geq N_\rho/8$, such that $\pi' \cap \kC' = \varnothing$. Using Lemma \ref{lem: crossing cluster}, there exists $c_4 > 0$ such that for all $N\geq 1$
\begin{align*}
   & \pr (\Lambda_{3N}(e) \text{ satisfies (i) but not (iv)}) \\
    & \leq \pr \Big(\exists z \in \Lambda_{3N}(e), \exists \text{ a crossing cluster } \kC' \subset \Lambda_{N_\rho/2}(z),\\
    &  \qquad \qquad \qquad \qquad \qquad \qquad \exists \, \pi' \in \kP(\Lambda_{N_\rho/4}(z)) \cap \sO_{*}(\Lambda_{C_*N_\rho/4}(z)),
    \Diam(\pi') \geq N_\rho/8: \pi' \cap \kC' = \varnothing \Big) \\
    & \leq c_4^{-1}N^{d}\exp(- c_4 N_{\rho}).
\end{align*}
Combining this estimate with \eqref{satisying i}, \eqref{Good box: proper ii}, and \eqref{(ii)} yields the desired result. 
\end{proof}

\subsubsection{Proof of Proposition \ref{Lem: effective radius}} 

Fix $e \in \cE$. Using the definition of $R_e$, we have for each $t \geq 2$, 
\begin{align*}
    \pr(R_e \geq t) \leq 1- \pr(\mathcal{V}_{t-1}^1(e) \cap \mathcal{V}_{t-1}^2(e)).
\end{align*}
Using this estimate and Lemma \ref{Lem: being bad annulus}, the result follows once we can prove 
\begin{align}\label{ane:subset}
    \{ \Lambda_{3N}(e) \text { is } \textbf{ good} \} \subset \kV^1_N(e) \cap \kV^2_N(e).
\end{align}
To this end, we assume that $ \Lambda_{3N}(e) \text{ is } \textbf{good}$. We first prove that  $  \kV^1_N(e)$ holds. 
 Let $\gamma_1, \gamma_2 \in \sO_{*}(\Lambda_{C_*N}(e)) \cap \sC(\aAn)$. For each $j \in \{1,2\}$, there exists a connected subpath $\pi_j \subset \gamma_j \cap \left\{ \Lambda_{2N+ \frac{N_\rho}{2}}(e) \setminus \Lambda_{2N -\frac{N_\rho}{2}}(e) \right\}$ satisfying 
\begin{align*}
 \forall j \in \{1,2\} \quad \pi_j \in \kP(\Lambda_{3N}(e)) \cap \sO_{*}(\Lambda_{C_*N}(e)), \quad \Diam(\pi_j) \geq N_\rho, \quad \text{d}_{\infty}(\pi_j, \partial \aAn) \geq 3N/4.
\end{align*}
 Then by Definition \ref{Def: good annulus} (iv), we have $\pi_1 \cap \kC \neq \varnothing$ and $\pi_2 \cap \kC \neq \varnothing$, with $\kC$ the cluster crossing all sub-boxes of side-length $N_{\rho}$ of $\Lambda_{3N}(e)$. Therefore, there exist $ u,v \in \aAn$ such that
 $ u \in \pi_1 \cap \kC$, $ v \in \pi_2 \cap \kC$, and ${\rm d}_{\infty} (\{u,v\}, \partial \aAn) \geq 3N/4$. 
 Moreover, since $\kC$ contains a crossing cluster in $\Lambda$ for all $\Lambda \in \kB_{N_{\rho}}(\Lambda_{3N}(e))$, we can find a sequence of vertices $(x_i)_{i=0}^h \subset \kC$ with $h \leq (6N/N_\rho)^d  \leq (100 \rho^2)^d$ such that
 \be{
 x_0=u, \, \, x_h=v, \quad \rmd_{\infty}(x_i,\partial \aAn) \geq N/2, \quad \|x_{i-1}-x_i\|_\infty \leq 2 N_\rho \quad \forall \, i \in [h].
 }
Remark further that $\rmD(x_{i-1},x_i)<\infty$, as $(x_i)_{i=0}^h \subset \kC$. Hence, it follows from Definition \ref{Def: good annulus} (iii) that $\rmD^{\aAn}(x_{i-1},x_i) \leq 4\rho N_\rho \leq N/4$. Therefore, $\kV^1_N(e)$ holds since
\be{
{\rm D}^{\aAn}(\gamma_1,\gamma_2) \leq \sum_{i=1}^{h} \rmD^{\aAn}(x_{i-1},x_{i}) \leq (6N/N_\rho)^d (4 \rho N_\rho) \leq C_*N. 
}
Next, we assert that $ \kV^2_N(e)$ holds. For  $x,y \in \Lambda_{3N}(e)$ with $\rmD^{\Lambda_{3N}(e)}(x,y)<\infty$, there exists an open path $\gamma \subset \Lambda_{3N}(e)$ joining $x$ to $y$. Thus, we can find a sequence $(z_i)_{i=0}^h \subset \gamma$ with $h \leq (6N/N_\rho)^d \leq  (100 \rho^2)^d$ such that $z_0 = x, z_h = y$  and 
\begin{align}\label{omega 2}
 \forall i \in [h], \quad \|z_{i-1}-z_i\|_\infty \leq 2 N_\rho.
  \end{align}
Then, by Definition \ref{Def: good annulus} (ii), $\rmD^{\Lambda_{4N}(e)}(z_{i-1},z_i) \leq 4 \rho N_\rho$.  Therefore, 
\be{
\rmD^{\Lambda_{4N}(e)}(x,y) \leq  \sum_{i=1}^{h} \rmD^{\Lambda_{4N}(e)}(z_{i-1},z_{i}) \leq h (4 \rho N_\rho) \leq C_*N, 
}
and thus $ \kV^2_N(e)$ holds. In summary, we have proved \eqref{ane:subset} and completed the proof of this lemma. 
 \qed 

\section{Lattice animals of dependent weight}\label{sec3}
 In the previous section, we derive that the impact of resampling an edge along a modified geodesic can be controlled by its effective radius. Our current goal is to estimate the total cost of resampling, i.e., the sum of the effective radii along a random set. Fortunately, although these effective radii are not independent, they exhibit relatively local dependence (Remark \ref{rem:weak}) and furthermore, satisfy the exponentially decaying tail probability (Proposition \ref{Lem: effective radius}).   Moreover, the theory of greedy lattice animals  can help us control the total weight of paths in locally dependent environments.

Given an integer $ M \geq 1$ and positive constants $a,A $, suppose that $(I_{e,M})_{ e \in \cE}$ is  a collection of Bernoulli random variables satisfying 
\begin{itemize}
    \item[(E1)] $(I_{e,M})_{ e \in \cE}$ are $aM-$dependent, i.e. for all $e \in \cE$, the variable $I_{e,M}$ is independent of all variables $(I_{e',M})_{ e' \not \in \Lambda_{aM}(e) }$. 
\item[(E2)]  
$$q_M := \sup_{e\in\cE} \E[I_{e,M}] \leq A M^{-d}.$$ 
\end{itemize}
For any set of edges $\gamma$, we define
\begin{align*}
 N(\gamma) := \sum_{e\in \gamma}I_{e,M}, \quad    N_{L,M} := \max_{\gamma \in \Xi_L} N(\gamma), 
\end{align*}
where for $L\geq 1$,
\begin{align*}
   \Xi_L := \{ \gamma: \gamma \text{ is a path in } \Lambda_L; \,  | \gamma| \leq L\}.
\end{align*}
\begin{lemma} \cite[Lemma 2.6]{nakajima2019first} \label{lemmaxbound} 
Let $M \geq 1, a, A>0$ and $(I_{e,M})_{e\in \cE}$  be a collection of random variables satisfying (E1) and (E2). Then, there exists a positive constant $C = C(a, A, d)$ such that
\begin{enumerate}[(i)]
   \item For all $L \geq 1$,
\begin{align*}
   \E[N_{L,M}] \leq C L q_M^{1/d}  M^{d+1}.
\end{align*}
\item If $t\geq C M^{d} \max\left\{1,MLq_M^{1/d}\right\}$, then
\begin{align*}
    \pr(N_{L,M} \geq t) < 2^d \exp(-t/(16M)^d).
\end{align*}
\end{enumerate}
\end{lemma}
\begin{remark}
    In \cite[Lemma 2.6]{nakajima2019first}, although the result is  stated for site-percolation, it  also holds true for edge-percolation by similar arguments, see for example \cite[Lemma 2.6] {can2023lipschitz}.
\end{remark}
We aim at extending this result to  general weight distributions. Let $a$ and $A$ be positive constants. Suppose that $(X_e)_{e \in \cE}$ is a  collection of non-negative random variables satisfying the following: for all integer $M \geq 1$
\begin{itemize}
    \item[(P1)] for all $e \in \cE$, the event $\{M-1 \leq X_e < M\}$ is independent of  $(X_{e'})_{e' \notin \Lambda_{aM}(e)}$, 
\item[(P2)] there exists a function $\phi: [0,\infty) \to [0,\infty)$  such that $ \phi(M) \leq  A M^{-d}$ and 
$$q_M = \sup_{e\in\cE} \pr(M-1 \leq  X_e < M) \leq \phi(M).$$ 
\end{itemize}
\begin{lemma} \label{lemkey} 
Let $X=(X_e)_{e \in \cE}$ be a family of random variables satisfying (P1) and (P2) and  let $f:[0,\infty) \to [0,\infty)$ be a function satisfying 
\begin{align} \label{h1}
\tag{H} B:=\sum_{M=1}^{\infty} \left(f_*(M)+f_*^2(M)+f_*^4(M) \right) M^{d+1}\phi(M)^{1/d}< \infty, \quad f_*(M) := \sup_{M-1 \leq x < M} f(x).
\end{align} 
Then there exists a positive constant $C=
C(a,A,B)$ such that the  following holds.  
\begin{enumerate}
\item [(i)] For all $L \geq 1$,
\begin{align*}
    \E \left[ \left( \max_{ \gamma \in \Xi_{L}} \sum_{e \in \gamma} f(X_{e})\right)^2 \right] 
 \leq C L^2.
\end{align*}
\item [(ii)] Let $\gamma$  be a random path  starting from $0$  in the same probability space as $X$. Then for all $L \geq 1$,
 \be{
  \E\left[ \left( \sum_{e \in \gamma} f(X_{e}) \right)^2 \right] \leq C L^2 + C \sum_{\ell \geq L} \ell^2 (\pp(| \gamma |= \ell))^{1/2}.
 }
 \item [(iii)] Let $m \geq 1$ and $\gamma$ be  a random path with ${\bf s}(\gamma) \in  \Lambda_m$   in the same probability space as $X$. Then for all $L \geq 1$,
 \be{
  \E\left[ \left( \sum_{e \in \gamma} f(X_{e}) \right)^2 \right] \leq  C (L+m)^2 + C \sum_{\ell \geq L} (\ell+m)^2 (\pp(| \gamma| = \ell))^{1/2}.
 }
 \item [(iv)]  
 Let $L \geq 1$ and $K \in \N$ such that $LM \phi(M)^{1/d} \geq 1$ for all $ 1 \leq M \leq K$. Assume  that $X_e < K$ for all $e \in \cE$. Then there exist constants $C_1,c_1 > 0$, independent of $L$ and $K$, such that for any  random  path $\gamma$ starting from $0$ in the same probability space as $X$,
\begin{align*}
 \pr \left(  \sum_{e \in \gamma} f(X_{e}) \geq C_1 L \right) \leq  \sum_{M=1}^{K} 2^d \exp(- c_1 L M \phi(M)^{1/d} )+ \pr(|\gamma| \geq L).
 \end{align*}
 \end{enumerate}
\end{lemma}
\begin{proof} 
We first prove (i). By Cauchy-Schwarz inequality, 
\ban{ \label{max}
\E\left[ \left( \max_{ \gamma \in \Xi_{L}} \sum_{e \in \gamma} f(X_{e}) \right)^2 \right] = \E\left[  \max_{ \gamma \in \Xi_{L}} \left( \sum_{e \in \gamma} f(X_{e}) \right)^2 \right] \leq \E\left[  \max_{ \gamma \in \Xi_{L}} | \gamma | \sum_{e \in \gamma} f^2(X_{e})  \right] \leq L\E\left[  \max_{ \gamma \in \Xi_{L}} \sum_{e \in \gamma} f^2(X_{e})  \right],
}
since $|\gamma | \leq L$ for all $\gamma \in \Xi_L$. For set of edges $\gamma$, we define
\begin{align*}
    A^\gamma_M: = \left|\{e \in \gamma: M-1 \leq X_e <  M\}\right|= \sum_{e \in \gamma} I_{e,M},
\end{align*}
where 
\begin{align*}
    I_{e,M}: =  \I(M-1 \leq X_e < M).
\end{align*}
Notice that,
\begin{align}\label{repre}
    \sum_{e \in \gamma}  f^2(X_{e}) \leq \sum_{M\geq 1} f^2_*(M)  A^\gamma_M ,
\end{align}
and hence
\begin{align} \label{boudnmaxsquare1}
    \E \left[ \max_{\gamma \in \Xi_L} \sum_{e \in \gamma}  f^2(X_{e})  \right] & \leq  \E \Bigg[\sum_{M\geq 1} f_*^2(M) \max_{\gamma \in \Xi_L}\sum_{e \in \gamma} I_{e,M}  \Bigg] = \sum_{M\geq 1} f_*^2(M)\E \left[ N_{L,M} \right],
\end{align}
where
$$  N_{L,M}= \max_{\gamma \in \Xi_L}\sum_{e \in \gamma} I_{e,M}.
$$
 By (P1) and (P2), the conditions (E1) and (E2) are satisfied for all $M \geq 1$. Now using Lemma \ref{lemmaxbound} (i), we obtain that for all $M \geq 1$,
\begin{align} \label{am>}
    \E [N_{L,M}] =\cO(1) L M^{d+1} \phi(M)^{1/d}.
\end{align}
This together with \eqref{boudnmaxsquare1} implies  that
\begin{align*}
 \E \Bigg[ \max_{ \gamma \in \Xi_{L}} \sum_{e \in \gamma} f^2(X_{e}) \Bigg]   & = \cO(1) L  \sum_{M \geq 1} f_*^2(M) M^{d+1}  (\phi(M))^{1/d} = \cO(L),
\end{align*}
where the second equality follows from
\begin{align}\label{Eq: condition f_*}
\sum_{M\geq 1} f_*^2(M) M^{d+1}\phi(M)^{1/d}< \infty, 
\end{align}
which is guaranteed by \eqref{h1}.
 Finally, combining the above estimate with \eqref{max}, we obtain (i).
 
 For (ii), we decompose
 \begin{align}\label{separate}
 \E\left[ \left( \sum_{e \in \gamma} f(X_{e}) \right)^2 \right] &  =\E\left[ \left( \sum_{e \in \gamma} f(X_{e}) \right)^2 \I(|\gamma| < L) \right] + \E\left[ \left( \sum_{e \in \gamma} f(X_{e}) \right)^2 \I(|\gamma| \geq L) \right] \notag \\
 &\leq  \E\left[ \left( \max_{\gamma \in \Xi_L} \sum_{e \in \gamma} f(X_{e}) \right)^2 \right] + \sum_{\ell=L}^{\infty} \E\left[|\gamma| \sum_{e \in \gamma} f^2(X_{e}) \I(|\gamma| =\ell) \right] \notag \\
 & \leq  \cO(L^2) + \sum_{\ell=L}^{\infty} \ell \E\left[ \max_{\sigma \in \Xi_\ell} \sum_{e \in \sigma} f^2(X_{e}) \I(|\gamma| =\ell) \right],
\end{align}
by using  (i). By \eqref{h1} again, the function $f^2$ satisfies \eqref{Eq: condition f_*}, i.e.,
\begin{align*}
\sum_{M\geq 1} f_*^4(M) M^{d+1}\phi(M)^{1/d}< \infty.
\end{align*}
Therefore, by first using Cauchy-Schwarz inequality and then applying (i) to $f^2$, we obtain
\begin{align*}
  \E\left[ \max_{\sigma \in \Xi_\ell} \sum_{e \in \sigma} f^2(X_{e}) \I(|\gamma| =\ell) \right] & \leq \E\left[ \left(\max_{\sigma \in \Xi_\ell} \sum_{e \in \sigma} f^2(X_{e}) \right)^2 \right]^{1/2}   \E\left[\I(|\gamma| =\ell) \right]^{1/2}   \leq \cO( \ell) (\pr(|\gamma| =\ell))^{1/2}.
\end{align*}
Combining the last two displays yields (ii). We can easily prove  (iii) by  using the same arguments as for (ii) and the fact that if $|\gamma| \leq t$ then $\gamma \in \Xi_{t+m}$ for all $t \geq 1$.

Finally, we show (iv).
Using  $X_e < K $ and a similar estimate as in \eqref{repre}, we have for any $C_1 >0$,
\begin{align} \label{pfxel}
\pr \left( \sum_{e \in \gamma} f(X_{e}) \geq C_1  L \right)  &\leq  \pr \left(  \sum_{e \in \gamma} f(X_{e}) \geq C_1 L,  |\gamma| \leq L \right) + \pr \left(  |\gamma| \geq L \right) \notag \\
  &  \leq \pr\left( \sum_{M=1}^{K}  f_*(M) N_{L,M} \geq C_1 L \right) + \pr \left(  |\gamma| \geq L \right).
\end{align}
Furthermore, the conditions (E1) and (E2) are satisfied for all $M \geq 1$. Let $C$ be the constant as in Lemma \ref{lemmaxbound}, and set
\be{
C_1 := C \sum_{M=1}^{\infty} f_*(M)  (\phi(M))^{1/d} M^{d+1}. 
}
Note that $ C_1 \in (0, \infty)$ by \eqref{h1}. Using Lemma \ref{lemmaxbound} (ii), 
\begin{align*}
\pr\left( \sum_{M=1}^{ K +1}  f_*(M) N_{L,M} \geq C_1 L \right) & \leq 
    \pr\left(\sum_{M=1}^{K}  f_*(M) N_{L,M} \geq \sum_{M=1}^{K}  C f_*(M) \phi(M)^{1/d} L M^{d+1} \right) \\
  & \leq \sum_{M=1}^{K} \pr\left(  N_{L,M} \geq C \phi(M)^{1/d} L M^{d+1} \right) \\
  & \leq \sum_{M=1}^{K} 2^d \exp(-c_1 LM\phi(M)^{1/d}),
\end{align*}
with $c_1=c_1(d)$ a positive constant. Combining this with \eqref{pfxel}, we derive (iv).
\end{proof}

Finally, we shall apply Lemma \ref{lemkey} to the sequence of truncated effective radii defined in \eqref{def: truncated radius}: $\hat{R}_e=\min\{C_*R_e, \log^2n\}$ for $ e \in \cE$. 
\begin{corollary} \label{corre}
   There exists a positive constant $C $ such that the following holds for all integers $n, m$ sufficiently large. 
\begin{enumerate}
\item [(i)] Let $\gamma$ be a random  path in the same probability space as $(R_e)_{e \in \cE}$, starting from $0$ and satisfying  $\pp( |\gamma| =\ell) \leq C \ell^{-7}$ for all integers $\ell\geq Cn$. Then 
 \be{
  \E\left[ \left( \sum_{e \in \gamma} \hat{R}_{e} \right)^2 \right] \leq  \E\left[ \left( \sum_{e \in \gamma} \hat{R}^2_{e} \right)^2 \right]
  \leq C n^2. 
 }
 \item [(ii)] Let $\gamma$ be a random path in the same probability space as $(R_e)_{e \in \cE}$ with ${\bf s}(\gamma) \in \Lambda_m$ for some integer $m\geq 1$. Suppose that $\pp(| \gamma| =\ell) \leq C \ell^{-7}$ for all integers $\ell\geq C m$. Then 
 \be{
 \E\left[  \sum_{e \in \gamma} \hat{R}_{e}^2 \right] \leq Cm.
 }
 \item [(iii)]   For any $c>0$, there exists $C_1= C_1(c)$ such that for all $\gamma$, a random path starting from $0$ in the same probability space as $(R_e)_{e \in \cE}$,
\begin{align*}
 \pr \left(  \sum_{e \in \gamma} \hat{R}_{e}^2 \geq  C_1 n \right) \leq C_1\exp\left(-\frac{n}{ C_1\log^{2(d+10)} n}\right)+ \pr(|\gamma| \geq c n).
 \end{align*}
 \end{enumerate}
\end{corollary}

\begin{proof} 
Let $f(x) = x^2$ and $\phi(x)= x^{-d^2-11d}$, so \eqref{h1} is verified. By Remark \ref{rem:weak} and Proposition \ref{Lem: effective radius}, the radii $(\hat{R}_e)_{e \in \cE}$ satisfy (P1) and (P2).   We take $X_e = \hat{R}_e$ for each $e \in \cE$. Thus, part (i) follows directly from Lemma \ref{lemkey} (ii) and part (ii) follows from Lemma \ref{lemkey} (iii) using the Cauchy-Schwarz inequality. Finally, applying Lemma \ref{lemkey} (iv) to $L= c n, K =  \lfloor \log^2 n \rfloor + 1$, we obtain (iii).
\end{proof}

\section{Subdiffusive concentration of $\rmT_n$} \label{sec5}
The proof strategy for Theorem \ref{subdiffconforT} relies on establishing  a connection between bounds on $\Var(e^{\lambda \rmT_n/2})$ and exponential concentration (see \cite[Lemma 4.1]{benaim2008exponential}). To achieve the required variance bound, we employ the Falik-Samorodnitsky inequality (Lemma \ref{fs}) to a martingale decomposition of the random variable $e^{\lambda \rmF_m}$, where $\rmF_m$ represents an averaged version of the passage time. This approach was first introduced by Benaïm and Rossignol in \cite{benaim2008exponential} and  subsequently utilized by Damron, Hanson, and Sosoe in \cite{damron2014subdiffusive}. Finally, we estimate the tails of the true passage time $\rmT_n$ based on those of $\rmF_m$.
\subsection{Variance bound via entropy inequality}
 Let us enumerate the edges $\cE$ as $\{e_1,e_2,\ldots\}$ and let $a,b \in \R \cup \{+\infty \}$. Assume that $(t_{e_i})_{i\geq 1}$ are i.i.d. random variables with the same distribution as
 $$ \zeta= p\delta_a+(1-p)\delta_{b}.$$
Let $g:\{a,b\}^\cE \to \R $ be a function of $(t_{e_i})_{i\geq 1}$. Fix $\lambda \in \R$ and define
 \begin{align*}
     G := G_\lambda := e^{\lambda g}.
 \end{align*}
  We  write 
\be{
G=G(t_{e_i},t_{e_i^c})
}
to emphasize $G$ is the function of the random variables $t_{e_i}$ and $t_{e_i^c} = (t_{e_j})_{j \neq i}$. We define the natural filtration of these random variables as 
\begin{align*}
    \cF_0:= \{\varnothing, \Omega\}, \quad  \cF_i := \sigma(t_{e_1},\ldots,t_{e_i}),
\end{align*}
for each $i \geq 1$. Now we consider the martingale increments
\begin{align*}
   \forall i \geq 1,\quad \Delta_i:= \E[G \mid \cF_i] - \E[G \mid \cF_{i-1}]= \E[G(t_{e_i},t_{e_i^c})-G(t'_{e_i},t_{e_i^c}) \mid \cF_{i}],
\end{align*}
where $t'_{e_i}$  is an independent copy of $t_{e_i}$. We will bound the variance of $G$ based on the following entropy inequality by Falik and Samorodnitsky \cite[Lemma 2.3]{falik2007edge}. 
\begin{lemma}[Falik-Samorodnitsky] \label{fs} If $\E[G^2] < \infty$ then
\begin{align}
    \sum_{i=1}^{\infty}  \ent[\Delta_i^2] \geq \var(G) \log \dfrac{\var(G)}{\sum_{i=1}^{\infty}  (\E[|\Delta_i|])^2},
\end{align}
where $\ent$ denotes the entropy operator: for any non-negative random variable $X$ with finite moment,
      \begin{align*}
          \ent[X] := \E \left[X \log \dfrac{X}{\E[X]}\right].
      \end{align*}
\end{lemma}
The following estimate on the total entropy is derived from the tensorization property of entropy and the log-Sobolev inequality for the Bernoulli distribution.
\begin{lemma}
\label{entropybound} Assume that $\E[G^4] < \infty$. Then, there exists a positive constant $C$ depending on $p$ such that
\begin{align}
       \sum_{i=1}^{\infty}  \ent[\Delta_i^2] \leq C    \sum_{i=1}^{\infty} \E[(G(b,t_{e_i^c})- G(a,t_{e_i^c}))^2].
\end{align}
\end{lemma}
The detailed proof of this lemma can be found in \cite{benaim2008exponential,damron2015sublinear}, or  \cite[Lemma 2.8]{dembin2022variance}.
\subsection{Proof of Theorem \ref{subdiffconforT} assuming Theorem \ref{thm2.3}} \label{Proof of 2.2}
 Instead of directly showing the subdiffusive concentration of $\ti$, we will employ a strategy inspired by Benjamini, Kalai, and Schramm in \cite{benjamini2003improved}, known as the BKS trick. This approach involves proving the subdiffusive concentration for a geometric average of passage time, a concept used    in \cite{benjamini2003improved,alexander2013subgaussian, damron2014subdiffusive,damron2015sublinear,bernstein2020sublinear}. 
Precisely, we define a spatial average of the first passage time,
\ben{  \label{mf}
    \mf := \frac{1}{|\Lambda_m|} \sum_{z \in \Lambda_m} \rmT^z,
}
where 
\begin{align*}
\forall z \in \Z^d, \quad \rmT^z := \rmT(z, z + n\e_1), \qquad 
m := \lfloor n^{1/4} \rfloor.
\end{align*}
To prove Theorem \ref{subdiffconforT}, it  suffices to show the following variance bound.
\begin{theorem}\label{thm2.3}
 There exists a constant $c>0$ such that
 \begin{align}
    \forall \, |\lambda| < \dfrac{1}{\sqrt{K}} ,\quad \var\left(e^{\lambda \mf}\right) \leq K\lambda^2 \E\left[e^{2\lambda \mf}\right] < \infty,
 \end{align}
 where $K = \dfrac{c n}{\log n}$.
 \end{theorem}
The above variance bound yields the following exponential concentration bound for $\mf$ via \cite[Lemma 4.1]{benaim2008exponential} or \cite[Theorem 3.10]{auffinger201750}.
\begin{corollary} \label{corrv}
There exists a positive constant $c$ such that
\begin{align}\label{largeoff}
    \pr\left(\left|\mf-\E[\mf]\right|\geq \sqrt{\frac{n}{\log n}} \kappa\right) \leq c^{-1} \exp(-c \kappa) \quad \forall \kappa \geq 0.
\end{align}
\end{corollary}
Next, we prepare a simple large deviation estimate for the first passage time, which will be used to compare $\rmT_n$ and $\rmF_m$. 
\begin{lemma}\label{geoTxy} Consider the truncated first passage percolation with the edge weight distribution $\zeta_p= p \delta_1 +(1-p) \delta_{\log^2n}$.
Then there exist positive constants $C,c$ such that for all $x, y \in \Z^d$ and  $t \geq  C \|x-y\|_{\infty}$,
\begin{align}
  \pr(\rmT(x,y) \geq t)\leq c^{-1} \exp(-c t/\log^2 n).     
\end{align}
\end{lemma}
\begin{proof} Observe that 
$
\rmT(u,v) \leq \log^2n \|u-v\|_1
$ for all  $u,v \in \Z^d$. Therefore, by the triangle inequality
    \begin{align*}
      \pr\left(\rmT(x,y) \geq t\right) & \leq \pr\left(\rmT(x,x^*)+ \rmT(y,y^*)+ \rmT(x^*,y^*) \geq t\right) \\
      & \leq \pr\left(\rmD^*(x,y) \geq t/2\right) + \pr\left(\|x-x^*\|_{1} \geq \tfrac{t}{4 \log^2 n}\right) + \pr\left(\|y-y^*\|_{1} \geq \tfrac{t}{4 \log^2 n}\right),
    \end{align*}
    where we recall that $z^*$ is the closest point of $z$ in the infinite cluster $\kC_{\infty}$. The  last  two terms are bounded by $c_1^{-1} \exp\left(\tfrac{-c_1 t}{\log^2 n}\right)$, for some positive constants $c_1$ using Lemma \ref{hole}, whereas by Lemma \ref{large deviation of D,0-x}, the first term is bounded by $c_2^{-1} \exp(-c_2t)$ when $t\geq  C \|x-y\|_\infty$ with some $C,c_2>0$.  Hence, the result follows.   
\end{proof}
\vspace{0.2 cm}
\noindent {\bf Proof of Theorem \ref{subdiffconforT}}. Since $\E[\mf] = \E[\ti]$,
\begin{align}\label{triatet}
    |\ti - \E[\ti] | = |\mf - \E[\ti]+ \ti - \mf| & =  |\mf - \E[\mf]+ \ti - \mf| \notag \\
    & \leq |\mf - \E[\mf]|+ |\ti - \mf|.
\end{align}
Thus, for all $M\geq 1$, using the union bound, we have
\begin{align} \label{dia}
    \pr\left(|\ti - \E[\ti]|\geq 4M\right) \leq \pr\left(|\mf - \E[\mf]| \geq 2M \right) + \pr\left(|\ti - \mf| \geq 2M\right).
\end{align}
By subadditivity property,
\begin{align}\label{tfm}
    |\ti - \mf| = \Big| \ti - \dfrac{1}{|\Lambda_m|} \sum_{z \in \Lambda_m} \rmT^z\Big| & \leq \dfrac{1}{|\Lambda_m|} \sum_{z\in \Lambda_m}| \rmT(0,n\e_1)-\rmT(z,z+n\e_1)| \notag \\
    &  \leq \dfrac{1}{|\Lambda_m|} \sum_{z\in \Lambda_m}( \rmT(0,z)+ \rmT(n\e_1,n\e_1+z)).
\end{align}
Observe that if the event $\Big\{ \tfrac{1}{|\Lambda_m|} \sum_{z\in \Lambda_m}( \rmT(0,z)+ \rmT(n\e_1,n\e_1+z)) \geq 2M \Big\} $ occurs,  
\begin{align}
    \max_{z\in \Lambda_m} \rmT(0,z) \geq M \text{ or }  \max_{z\in \Lambda_m}\rmT(n\e_1,n\e_1+ z) \geq M.
\end{align}
Combining this with union bound, it yields that
\begin{align}\label{maxbound}
     \pr &\left(\frac{1}{|\Lambda_m|}  \sum_{z\in \Lambda_m}( \rmT(0,z)+ \rmT(n\e_1,n\e_1+z))  \geq 2M \right) \notag \\
   & \qquad \qquad \qquad  \leq \pr\left(\max_{z\in \Lambda_m} \rmT(0,z) \geq M \right) + \pr \left( \max_{z\in \Lambda_m} \rmT(n\e_1,n\e_1+ z) \geq M \right) \notag \\
    & \qquad \qquad \qquad  = 2 \pr \left( \max_{z\in \Lambda_m}\rmT(0,z) \geq M \right)   \leq 2 |\Lambda_m| \max_{z\in \Lambda_m} \pr\left(\rmT(0,z) \geq M \right),
\end{align}
where the equality follows from translation invariance.

Let $ M = \tfrac{1}{4}\sqrt{\frac{n}{\log n}} \kappa$. Since $m=o(M)$, Lemma \ref{geoTxy} shows that
\begin{align}
    \max_{z \in \Lambda_m} \pr\left(\rmT(0,z) \geq M \right) 
    \leq c^{-1} \exp(-c M/\log^2 n),
\end{align}
for some constant $c >0$. Using this estimate, \eqref{tfm}, and \eqref{maxbound} yields
\be{
    \pr\left(|\ti- \mf| \geq \dfrac{\kappa}{2} \sqrt{\tfrac{n}{\log n}}\right) \leq  \cO(m^{d}) \exp\Big(-c\tfrac{\sqrt{n}}{4\sqrt{\log^5 n}} \kappa\Big).
    }
 Combining this with Corollary \ref{corrv} and \eqref{dia}, it follows that  
\begin{align}
\mathbb{P}\left(|\ti-\mathbb{E} [\ti]| \geq \sqrt{\frac{n}{\log n}} \kappa\right) \leq c^{-1} \exp(-c \kappa) \qquad \forall \kappa \geq 0,
\end{align}
for some positive constant $c$. \qed 

\vspace{0.2 cm}

\noindent {\bf Proof of Theorem \ref{thm2.3}}. 
 According to the Lemma \ref{fs}, the variance bound is induced by two crucial factors: the estimate of total influence and total entropy.
 These keys are presented in the following results. From now on, we set
 \begin{align*}
     G= e^{\lambda F_m}, \quad \Delta_i = \E[G|\cF_i] - \E[G|\cF_{i-1}], \quad \forall i \geq 1.
 \end{align*}
\begin{proposition}\label{thminfluence}
There exists a positive constant $C$ such that
\begin{align*}
\sum_{i=1}^{\infty} (\E[|\Delta_i|])^2 \leq C \lambda^2 \E\big[e^{2\lambda \rmF_m}\big] n^{(9-d)/8} \quad \forall \, \lambda \leq 1/\log^2n.
\end{align*}
\end{proposition}

\begin{proposition} \label{enb}
 There exists a positive constant $C$ such that
\begin{align}
  \sum_{i=1}^{\infty}  \ent[\Delta_i^2] & \leq C\lambda^2 n \E\big[e^{2\lambda \mf}\big]\quad \forall \, |\lambda| \leq  \tfrac{1}{\log^{2(d+11)} n}.
\end{align}
\end{proposition}
\noindent By Lemma \ref{fs},  Propositions \ref{thminfluence} and \ref{enb}, we have
\begin{align}\label{bound}
    \Var\big(e^{\lambda \rmF_m}\big) \leq C \left(\log\frac{\var(e^{\lambda \rmF_m})}{C \lambda^2n^{(9-d)/8} \E\left[e^{2\lambda \mf} \right]}\right)^{-1}  \lambda^2 n \E\big[e^{2\lambda \mf} \big].
\end{align}
 We can assume that there exists $\lambda$ with $|\lambda| \leq \tfrac{1}{\log^{2(d+11)} n}$ such that
\begin{align} \label{var}
    \var(e^{\lambda \rmF_m}) \geq C \lambda^2 n^{15/16}    \E\left[e^{2\lambda \mf} \right],
\end{align}
since otherwise there is nothing left to prove. By \eqref{bound} and \eqref{var},  there exists a constant $c>0$ such that for any $ |\lambda| \leq \tfrac{1}{\log^{2(d+11)} n} $,
\begin{align*}
     \var\big(e^{\lambda \rmF_m}\big) \leq c \lambda^2  \dfrac{n}{\log n } \E\big[e^{2\lambda \mf} \big].
\end{align*}
This concludes the proof of Theorem \ref{thm2.3} by substituting $\lambda/2$ for $\lambda$. \qed  

In the rest of Section \ref{sec5}, we prove Propositions \ref{thminfluence} and \ref{enb} in  subsections \ref{sub:bin} and \ref{sub:en}, respectively. 

\subsection{Bound on the total influence: Proof of Proposition \ref{thminfluence}} \label{sub:bin}

Proposition \ref{thminfluence} is a direct consequence of the following lemma  with the notice that $m = \lfloor n^{1/4} \rfloor$.
\begin{lemma} \label{proinfluence1} There exists a  constant $C>0$ such that the following holds.
\begin{itemize}
    \item [(i)] \begin{align}
    \sup_{i\geq 1} \E[|\Delta_i|] \leq  C |\lambda| m^{(1-d)/2} \big(\E\big[e^{2\lambda \mf}\big]\big)^{1/2} \quad \forall \, \lambda \leq 1/\log^2n.
\end{align}
\item [(ii)] \begin{align}
   \sum_{i=1}^{\infty} \E[|\Delta_i|] \leq C|\lambda| n \big(\E\big[e^{2\lambda \mf} \big]\big)^{1/2} \quad \forall \, \lambda \leq 1/\log^2n.
\end{align}
\end{itemize}
\end{lemma}

\subsubsection{Proof of  Lemma  \ref{proinfluence1} (i)} Fix $i\geq 1$ and consider
\begin{align}
    \Delta_i = \E[G|\cF_i]-\E[G|\cF_{i-1}] = \E \left[ G(t_{e_i},t_{e_i^c})-G(t'_{e_i},t_{e_i^c})|\cF_{i} \right],
\end{align}
where recall that $t'_{e_i}$ is the independent copy of $t_{e_i}$. We have
\begin{align*}
    \E\left[|\Delta_i|\right]\leq \E\left[|G(t'_{e_i},t_{e_i^c})-G(t_{e_i},t_{e_i^c})| \right] = 
 \E\left[\left|e^{\lambda \mf(t'_{e_i},t_{e_i^c})}- e^{\lambda \mf(t_{e_i},t_{e_i^c})}\right|\right].
\end{align*}
Furthermore, using the inequality that $|e^{\lambda a} -e^{\lambda b}| \leq |\lambda| (e^{\lambda a} + e^{\lambda b})  |a-b|$, we get
\begin{align}\label{delta2}
    \E[|\Delta_i|]  & \leq  |\lambda| \E\left[\left(e^{\lambda \mf(t'_{e_i},t_{e_i^c})}+e^{\lambda \mf(t_{e_i},t_{e_i^c})}\right)  \left|\mf(t'_{e_i},t_{e_i^c})-\mf(t_{e_i},t_{e_i^c}) \right|\right] \notag\\
    &  = 2 |\lambda| \E\left[\left(e^{\lambda \mf(t'_{e_i},t_{e_i^c})}+e^{\lambda \mf(t_{e_i},t_{e_i^c})}\right)
   \left(\mf(t'_{e_i},t_{e_i^c})-\mf(t_{e_i},t_{e_i^c}) \right)\I(t'_{e_i} > t_{e_i})
    \right] \notag \\
    & \leq 8  |\lambda| \E\left[e^{\lambda \mf(t_{e_i},t_{e_i^c})}
   \left(\mf(t'_{e_i},t_{e_i^c})-\mf(t_{e_i},t_{e_i^c}) \right)\I(t'_{e_i} > t_{e_i})
    \right] \notag \\
    & \leq 8 |\lambda| \E\left[e^{\lambda \mf(t_{e_i},t_{e_i^c})}
   \left(\mf(\log^2 n,t_{e_i^c})-\mf(1,t_{e_i^c}) \right)\I(t_{e_i} =1) \right],
\end{align}
where for the second inequality we have used that if $t'_{e_i} > t_{e_i}$ then $\lambda\mf(t'_{e_i},t_{e_i^c}) \leq \lambda(\mf(t_{e_i},t_{e_i^c}) + t'_{e_i}) \leq \lambda \mf(t_{e_i},t_{e_i^c}) +1 $ since $ \lambda \leq 1 / \log^2 n$ and $t'_{e_i} \in \{1, \log^2 n\}$. For each $z \in \Z^d$, let $\gamma_z$ be a geodesic of $\rmT^z(t_{e_i},t_{e^c_i})$. 
By Proposition \ref{prop1},
\begin{align*}
  (\rmT^z (\log^2 n,t_{e^c_i})- \rmT^z(1,t_{e^c_i}))\I(t_{e_i}=1)  \leq  ( \log^2 n \I(\cU_{z,e_i})+ \hat{R}_{e_i}) \I(e_i \in \gamma_z),
\end{align*}
where recall that
$$
 \mathcal{U}_{z,e_i}= \{ R_{e_i} \geq r_{z,e_i}\}, \quad r_{z,e_i} = \tfrac{1}{3} \|e_i-z\|_{\infty} \wedge \|e_i-(z+n\e_1)\|_{\infty}, \quad \hat{R}_{e_i} = \min \{C_* R_{e_i}, \log^2 n\}.
$$
Therefore,
\ba{ \label{boundofpypass}
      (\mf(\log^2 n,t_{e^c_i}) - \mf(1,t_{e^c_i}))\I(t_{e_i}=1) 
    &= \frac{1}{|\Lambda_m|} 
\sum_{z\in \Lambda_m} (\rmT^z (\log^2 n,t_{e^c_i})- \rmT^z (1,t_{e^c_i}))\I(t_{e_i}=1) \\ 
&\leq  \frac{1} {|\Lambda_m|} 
 \sum_{z\in \Lambda_m} \big( \log^2 n \I(\cU_{z,e_i})+ \hat{R}_{e_i} \big) \I(e_i \in \gamma_z).
} 
Observe that if the event $\cU_{z,e_i} \cap \{r_{z,e_i} \geq \log^3 n\} $ occurs, then $C_*R_{e_i} \geq \log^2 n$ and so $\hat{R}_{e_i}=\log^2 n$. Therefore, the above estimate implies that 
\ben{ \label{mfai}
0 \leq (\mf(\log^2 n,t_{e^c_i}) - \mf(1,t_{e^c_i}))\I(t_{e_i}=1)  \leq A_i,
}
where
 \begin{align} \label{siai}
 A_i & =   \frac{1} {|\Lambda_m|} 
 \sum_{z\in \Lambda_m}  \big( 2\hat{R}_{e_i} + \log^2 n  \I(r_{z,e_i} \leq \log^3 n) \big) \I(e_i \in \gamma_z).
 \end{align}
 Combining this with \eqref{delta2} and Cauchy-Schwarz inequality yields
\ban{
\E[|\Delta_i|]  & \leq    8|\lambda| \E\left[e^{\lambda \mf}  A_i \right]\label{edi} \\
& \leq  8|\lambda|  \E\left[e^{2\lambda \mf}  \right]^{1/2}   \E\left[ A_i^2 \right]^{1/2}. \label{bound56}
}
Here for the first line, we remark that $\mf(t_{e_i},t_{e_i^c})=\mf$.

Next, we will estimate  $\E[A_i^2]$. Notice that for all edges $e \in \cE$ and $\Lambda \subset \Z^d$,
\ben{ \label{borz}
\left|\{z \in \Lambda: r_{z,e} \leq t\} \right| \leq \left|\{z \in \Z^d: r_{z,e} \leq t\} \right|=\cO(t^d).
}
 Therefore,
\begin{align}
   \frac{1}{|\Lambda_m|} \sum_{z\in \Lambda_m }  \log^2 n \I(r_{z,e_i} \leq \log^3  n  ) \leq  \frac{ \cO(\log^{3d+2}n)}{ | \Lambda_m|}=\cO(m^{1-d}),
\end{align} 
since $m= \lfloor n^{1/4} \rfloor$. Thus, by Cauchy-Schwarz inequality, 
\ban{
 A_i^2 & \leq 2 \left(\frac{1} {|\Lambda_m|} 
 \sum_{z\in \Lambda_m} 2\hat{R}_{e_i}  \I(e_i \in \gamma_z) \right)^2 + 2 \left(\frac{1}{|\Lambda_m|} 
 \sum_{z\in \Lambda_m} \log^2 n \I(r_{z,e_i} \leq \log^3 n) \I(e_i \in \gamma_z)\right)^2 \notag \\
 & \leq  \frac{8} {|\Lambda_m|} 
 \sum_{z\in \Lambda_m} \hat{R}_{e_i}^2 \I(e_i \in \gamma_z)
 + \cO(m^{2-2d}).
}
Combining this estimate with the translation invariance, we have
\begin{align}\label{componentbound}
\E[A_i^2]& \leq  \frac{8} {| \Lambda_m|} 
\E \left[ \sum_{z\in  \Lambda_m} \hat{R}_{e_i-z}^2 \I(e_i-z \in \gamma_0) \right] + \cO(m^{2-2d}) \notag \\
& = \frac{8} {| \Lambda_m|} 
 \E \left[\sum_{e \in \gamma} \hat{R}_{e}^2 \right] + \cO(m^{2-2d}),
\end{align}
where
\begin{align*}
    \gamma := \gamma_0  \cap \{e_i-\Lambda_m\}, \quad  \{e_i-\Lambda_m\} :=\{e'=(x_{e_i}-z,y_{e_i}-z): z \in \Lambda_m\}.
\end{align*}
Let $u$ and $v$ be the first and last points of $\gamma_0$ intersecting with $ \{e_i-\Lambda_m\}$. Let $\gamma'$ be the subpath of $\gamma_0$ from $u$ to $v$, so that $\gamma \subset \gamma'$. Observe that if $ |\gamma'| \geq \ell$ then we can find $x,y$ in $V(e_i,m)$, the vertex set of $e_i-\Lambda_m$, such that $\rmT(x,y) \geq  \ell$. Therefore, using the union bound and Lemma \ref{geoTxy} we have for all $\ell \geq  C m$ with $C$ a sufficiently large constant, 
\begin{align} \label{phibound}
   \pr(|\gamma'| \geq \ell)  \leq  \pr\left(\exists \, x,y \in V(e_i,m): \rmT(x,y) \geq \ell \right) & \leq \cO(1)m^{2d} \max_{x,y \in V(e_i,m)} \pr\left( \rmT(x,y) \geq \ell \right) \notag \\
   & \leq \cO(1) \exp\left(-\frac{c \ell}{\log^2 n}\right),
\end{align}
with some positive constant $c>0$.
Here, notice that to apply  Lemma \ref{geoTxy}, we have used $\|x-y\|_{\infty} \leq 2m$ for all $x,y \in V(e_i,m)$. The above estimate verifies the condition in Corollary \ref{corre} (ii) and thus 
\begin{align}\label{boundre}
\E \left[ \sum_{e \in \gamma}
\hat{R}_{e}^2 \right] \leq \E \left[ \sum_{e \in \gamma'}
\hat{R}_{e}^2 \right]  = \cO(m).
\end{align}
 Combining \eqref{componentbound} and \eqref{boundre} yields that for all $i\geq 1$,
\begin{align} \label{boundb2}
    \E[A_i^2] = \cO(m^{1-d}).
\end{align}
Finally, we conclude from \eqref{bound56} and \eqref{boundb2} that
\begin{align}
  \sup_{i\geq 1}  \E[|\Delta_i|] \leq \cO(1) |\lambda| \big(\E\big[e^{2\lambda \mf}\big]\big)^{1/2} m^{(1-d)/2},
\end{align}
and the result follows.  \qed 
\subsubsection{Proof of  Lemma  \ref{proinfluence1} (ii)}
 Using \eqref{edi} and Cauchy-Schwarz inequality, we obtain that
\begin{align} \label{boundsum}
  \sum_{i=1}^\infty   \E[|\Delta_i|]   \leq  8|\lambda|  \E\left[e^{\lambda \mf} \sum_{i=1}^{\infty} A_i \right]  
  & \leq  8|\lambda|  \left(\E\left[e^{2\lambda \mf}\right]\right)^{1/2} \left(\E \left[ \left(\sum_{i=1}^{\infty} A_i \right)^2\right] \right)^{1/2},
\end{align}
where $A_i$ is defined as in \eqref{siai}.  Notice that 
\ba{
\sum_{i=1}^{\infty} A_i &= \frac{1}{|\Lambda_m|} 
 \sum_{z\in \Lambda_m} \sum_{i=1}^{\infty} \big(  2\hat{R}_{e_i} + \log^2 n  \I(r_{z,e_i} \leq \log^3 n) \big) \I(e_i \in \gamma_z)\\
 &= \frac{2} {|\Lambda_m|} 
 \sum_{z\in \Lambda_m} \sum_{e \in \gamma_z}  \hat{R}_{e} + \frac{\log ^2 n} {|\Lambda_m|} 
 \sum_{z\in \Lambda_m} \sum_{e \in \gamma_z }   \I(r_{z,e} \leq \log^3 n) \\
 &= \frac{2} {|\Lambda_m|} 
 \sum_{z\in \Lambda_m} \sum_{e \in \gamma_z}  \hat{R}_{e} +  \cO(\log^{3d+2} n),
}
by using \eqref{borz}. Therefore, thanks to Cauchy-Schwarz inequality,
\ban{ \label{sumai}
\Big(\sum_{i=1}^{\infty} A_i \Big)^2 \leq   \frac{8}{|\Lambda_m|}  \sum_{z\in \Lambda_m}  \Big(\sum_{e \in \gamma_z} \hat{R}_{e} \Big)^2 + \cO(\log^{6d+4} n).
}
 It follows from Lemma \ref{geoTxy} that
 \begin{align*}
     \pr(|\gamma_{z}| \geq C n) \leq c^{-1} \exp(-c n/\log^2 n).
 \end{align*}
 Then applying Corollary \ref{corre} (i) gives
\begin{align}\label{ei}
  \E \left[  \left(\sum_{e \in \ega_{z}}   \hat{R}_{e} \right)^2 \right]= \cO(n^2).
\end{align}
Combining \eqref{sumai}  with \eqref{ei} yields 
\begin{align} \label{esai}
    \E \left[ \left(\sum_{i=1}^{\infty} A_i \right)^2\right] = \cO(n^2),
\end{align}
which together with  \eqref{boundsum} implies  that 
\begin{align*}
    \sum_{i=1}^\infty   \E[|\Delta_i|]  & \leq \cO(1) |\lambda| n \big(\E\big[e^{2\lambda \mf} \big]\big)^{1/2}.
\end{align*}
\qed 
\subsection{Entropy bound: Proof proposition \ref{enb} } \label{sub:en} Observe that for each $i \geq 1$,
\begin{align*}
    \E\Big[(G(t'_{e_i},t_{e_i^c})- G(t_{e_i},t_{e_i^c}))^2\Big] & = 
     2 \E\Big[(G(t'_{e_i},t_{e_i^c})- G(t_{e_i},t_{e_i^c}))^2\I(t'_{e_i}> t_{e_i})\Big]\\
    & = 2 \E\Big[(G(\log^2 n,t_{e_i^c})- G(1,t_{e_i^c}))^2\I(t'_{e_i}=\log^2n,t_{e_i}=1)\Big] \\
    & = 2 p(1-p) \E\Big[(G(\log^2 n,t_{e_i^c})- G(1,t_{e_i^c}))^2\Big],
\end{align*}
where for the last line we have used the independence of $(G(\log^2 n,t_{e_i^c})- G(1,t_{e_i^c}))$, $t'_{e_i}$, and $t_{e_i}$.
Therefore, using Lemma \ref{entropybound}, the total entropy is bounded by
\begin{align} \label{78}
       \sum_{i=1}^{\infty}  \ent[\Delta_i^2]  &\leq C    \sum_{i=1}^{\infty} \E\Big[(G(\log^2 n,t_{e_i^c})- G(1,t_{e_i^c}))^2\Big] \notag \\
       & = \frac{C}{2p(1-p)} \sum_{i=1}^{\infty} 
       \E\Big[(G(t'_{e_i},t_{e_i^c})- G(t_{e_i},t_{e_i^c}))^2\Big]
      \notag \\ 
       & =   \frac{C}{2p(1-p)}   \sum_{i=1}^{\infty} \E\Big[\Big(e^{\lambda \rmF_m(t'_{e_i} ,t_{e_i^c})}- e^{\lambda \rmF_m(t_{e_i},t_{e_i^c})}\Big)^2  \Big] \notag \\
      & \leq \frac{C}{2p(1-p)}|\lambda|^2 \sum_{i=1}^{\infty} \E\Big[\Big(e^{\lambda \mf(t'_{e_i},t_{e_i^c})} + e^{\lambda \mf(t_{e_i},t_{e_i^c})}\Big)^2 \big(\mf(t'_{e_i} ,t_{e_i^c})-\mf(t_{e_i},t_{e_i^c})\big)^2\Big] \notag \\
      & \leq \frac{C}{p(1-p)}|\lambda|^2 \sum_{i=1}^{\infty} \E\Big[\Big(e^{2\lambda \mf(t'_{e_i},t_{e_i^c})} + e^{2\lambda \mf(t_{e_i},t_{e_i^c})}\Big) \big(\mf(t'_{e_i} ,t_{e_i^c})-\mf(t_{e_i},t_{e_i^c})\big)^2\Big] \notag \\
      & = \frac{2C}{p(1-p)}|\lambda|^2 \sum_{i=1}^{\infty} \E\Big[\Big(e^{2\lambda \mf(t'_{e_i},t_{e_i^c})} + e^{2\lambda \mf(t_{e_i},t_{e_i^c})}\Big) \big(\mf(t'_{e_i},t_{e_i^c})-\mf(t_{e_i},t_{e_i^c})\big)^2\I(t'_{e_i} > t_{e_i})\Big].
\end{align}
Here for the second inequality we have used the bound  $|e^{\lambda a}-e^{\lambda b}| \leq |\lambda| (e^{\lambda a}+e^{\lambda b})|a-b|$, while the third inequality follows from the Cauchy-Schwarz inequality. Now by using the similar argument as in \eqref{delta2}, the $i$-th summand in \eqref{78} is bounded above by
\begin{align*}
    9\E\Big[ e^{2\lambda \mf(t_{e_i},t_{e_i^c})} \big(\mf(\log^2 n,t_{e_i^c})-\mf(1,t_{e_i^c})\big)^2\I( t_{e_i} =1)\Big].
\end{align*}
Substituting this into \eqref{78}, we obtain
\begin{align*}
    \sum_{i=1}^{\infty}  \ent[\Delta_i^2] & \leq \frac{18C}{p(1-p)}|\lambda|^2  \sum_{i=1}^{\infty} \E\left[e^{2\lambda \mf} \left(\mf(\log^2n ,t_{e_i^c})-\mf(1,t_{e_i^c})\right)^2 \I( t_{e_i} =1)\right].
\end{align*}
Combining this with \eqref{mfai}, we get
\begin{align}
    \sum_{i=1}^{\infty}  \ent[\Delta_i^2]
    &  \leq \cO(1) \lambda^2 \sum_{i=1}^{\infty} \E\left[e^{2\lambda \mf} A_i^2\right] = \cO(1) \lambda^2 \E\left[e^{2\lambda \mf} \sum_{i=1}^{\infty} A_i^2\right].
\end{align}
By Cauchy-Schwarz inequality and \eqref{borz},
\ban{ \label{sais}
\sum_{i=1}^{\infty} A^2_i & = \frac{1} {|\Lambda_m|^2} \sum_{i=1}^{\infty} \left(
 \sum_{z\in \Lambda_m}  \big( 2\hat{R}_{e_i} + \log^2 n  \I(r_{z,e_i} \leq \log^3 n) \big) \I(e_i \in \gamma_z)\right)^2 \notag \\
 & \leq  \frac{1} {|\Lambda_m|} \sum_{i=1}^{\infty}
 \sum_{z\in \Lambda_m}  \big( 2\hat{R}_{e_i} + \log^2 n  \I(r_{z,e_i} \leq \log^3 n) \big)^2 \I(e_i \in \gamma_z) \notag \\ 
& \leq  \frac{1} {| \Lambda_m|} 
 \sum_{z\in \Lambda_m} \sum_{i=1}^{\infty} \big(  8\hat{R}^2_{e_i} + 2\log^4 n  \I(r_{z,e_i} \leq \log^3 n) \big) \I(e_i \in \gamma_z) \notag \\
 &= \frac{8} {| \Lambda_m|} 
 \sum_{z\in  \Lambda_m} \sum_{e \in \gamma_z}  \hat{R}^2_{e} + \frac{2\log ^4 n} {| \Lambda_m|} 
 \sum_{z\in  \Lambda_m} \sum_{e \in \gamma_z }   \I(r_{z,e} \leq \log^3 n) \notag  \\
 &= \frac{8} {| \Lambda_m|} 
 \sum_{z\in  \Lambda_m} Y_z  +  \cO(\log^{3d+4} n),
}
where 
\be{
Y_z:=\sum_{e \in \gamma_z}  \hat{R}^2_{e}.
}
Combining the last two estimates, we obtain  
\begin{align}\label{ent}
    \sum_{i=1}^{\infty}  \ent[\Delta_i^2] & \leq   \frac{\cO(1) \lambda^2} {| \Lambda_m|} \sum_{z\in  \Lambda_m} \E\big[e^{2\lambda \mf}  Y_z\big] + \cO(\log^{3d+4} n)  \lambda^2 \E\big[e^{2\lambda \mf}\big] .
\end{align}
By Lemma \ref{geoTxy} and the union bound, there exist positive constants $C',c$ such that for all $z \in \Z^d$ and $t \geq C'n$,
\begin{align} \label{largedeviationofTz}
    \max \{\pr(\rmT^z \geq t ), \pr(\rmF_m\geq t) \} \leq c^{-1} \exp\left(- \frac{ct}{\log^2 n}\right). 
\end{align}
Particularly, 
 \begin{align} \label{lagz}
     \pr(|\gamma_z| \geq t) \leq \pr(\rmT^z \geq t)  \leq c^{-1}  \exp\left(-\frac{ct}{\log^2 n}\right).
 \end{align}
Using this estimate and Corollary \ref{corre} (iii), we obtain that
 \begin{align} \label{lareyz}
    \pr (Y_{z} \geq C n) \leq C\exp\left(-\frac{n}{ C \log^{2(d+10)} n}\right),
 \end{align}
 for some positive constant $C$. Moreover, 
 \begin{align}\label{e2fy}
      \E\Big[  e^{2\lambda \mf}  Y_{z} \Big] &\leq C n  \E\Big[  e^{2\lambda \mf} \Big] +  \E\Big[  e^{2\lambda \mf} Y_{z} \I( Y_{z} \geq C n) \Big] \notag \\
      &\leq C n  \E\Big[  e^{2\lambda \mf} \Big] + (\E[  e^{8\lambda \mf} ])^{1/4}  (\pr( Y_{z} \geq C n))^{1/4} (\E[Y_{z}^2])^{1/2},
 \end{align}
 thanks to Cauchy-Schwarz inequality. By \eqref{lagz} and Corollary \ref{corre} (i), 
\begin{align} \label{eyz2}
\E[Y_z^2]  \leq \cO(n^2). 
\end{align}
 It follows from \eqref{largedeviationofTz} that for all  $\lambda \leq \frac{1}{\log^{2(d+11)} n}$, 
 \be{
\E[  e^{8\lambda \mf} ]  \leq  2 \exp\left(\frac{8 C' n}{\log^{2(d+11)} n}\right),
 }
 where  $C'$ is the constant in \eqref{largedeviationofTz}.
Putting this estimate, \eqref{lareyz}, \eqref{e2fy}, and \eqref{eyz2} together, we have for all  $\lambda \leq \frac{1}{\log^{2(d+11)} n}$, 
\begin{align}
    \E\Big[  e^{2\lambda \mf} Y_{z} \I( Y_{z} \geq C n) \Big] \leq C\exp\left(-\frac{n}{2C\log^{2(d+10)} n} \right).
\end{align}
Thanks to \eqref{largedeviationofTz} again, we have  for all $  -\frac{1}{\log^{2(d+11)}n} \leq \lambda < 0$,
 $$ \E\big[  e^{2\lambda \mf}\big] \geq  \E\big[  e^{2\lambda \mf} \I(\mf \leq C'n)\big] \geq \frac{1}{2}\exp\left(-\frac{2C'n}{\log^{2(d+11)}n} \right),$$
 with $C'$ being the constant defined as in \eqref{largedeviationofTz}. Combining the last two display equations with \eqref{e2fy} yields for all $ |\lambda| \leq \frac{1}{\log^{2(d+11)}n}$,
\begin{align*}
      \E\big[  e^{2\lambda \mf}  Y_{z} \big] & \leq C n  \E\big[  e^{2\lambda \mf} \big] +  C\exp\left(-\frac{n}{2C\log^{2(d+10)} n} \right)  \leq 2 C n\E\big[  e^{2\lambda \mf} \big],
 \end{align*}
 which together with \eqref{ent}  implies the desired result. \qed 

\section{Comparison of the graph distance and the  truncated first passage time} \label{sec6}
In order to prove  Theorem \ref{prodecre}, our aim  is to upper bound the difference between $\rmD^*_n$ and $\rmT(0^*, (n\e_1)^*)$ by the total size of a collection of bypasses avoiding all closed ($\log^2 n$-weight) edges along $\gamma$--a geodesic of $\rmT(0^*, (n\e_1)^*)$. More precisely, we introduce an inductive process that reveals a set of closed edges $\Gamma \subset \clo(\gamma)$ that are separable, as described by  property ${\rm (ii)}$ below.   Consequently, the discrepancy can be  bounded via the total sum of effective radii associated with  edges in $\Gamma$. Recall that $C_*$ is the constant and  $(R_e)_{e \in \cE}$ are the effective radii defined as in Section \ref{seckey}.
\begin{proposition}\label{lemdis}
 Let $\gamma$ be a geodesic of $\rmT(0^*, (n\e_1)^*)$. Suppose that $\{ 0^*,(n \e_1)^*\} \not \subset \Lambda_{3R_e}(e)$ for all $e \in \mathrm{clo}(\gamma)$. Then the following holds.  
\begin{enumerate}[(i)]
    \item For all $e \in  \mathrm{clo}(\gamma)$, $R_e \geq \tfrac{\log^2 n}{2C_*}$.
    \item There exists a subset $\Gamma \subset \clo(\gamma)$ such that
\begin{enumerate}
    \item[(iia)]
    for all $e \neq e' \in \Gamma $,
    $$ \|e-e'\|_{\infty} \geq \max\{R_e,R_{e'}\},$$
    \item[(iib)]
    \begin{align*}
        |\nt - \rmT(0^*, (n\e_1)^*)| \leq 2C_* \sum_{e \in \Gamma} R_e.
    \end{align*}
\end{enumerate}
\end{enumerate}
\end{proposition}
\begin{proof}

If $e \in \clo(\gamma)$ satisfies $ \{0^*,(n \e_1)^*\} \not \subset \Lambda_{3R_e}(e)$, we can apply Proposition \ref{Lem: application of er 2} to $\gamma \in \sO_{*}(0^*,(n\e_1)^*)$ and $e$, obtaining another path $\gamma_e \in \sO_{*}(0^*,(n\e_1)^*)$ such that $\gamma_e \setminus \gamma$ is open, and $|\gamma_e \setminus \gamma| \leq 2 C_*R_e$. Hence, 
\begin{align*}
   \rmT(\gamma) \leq   \rmT(\gamma_{e}) \leq \rmT(\gamma) + \rmT(\gamma_{e} \setminus \gamma)-t_{e} \leq\rmT(\gamma) + 2C_* R_{e}- \log^2 n,
 \end{align*}
which implies that $R_{e} \geq  \log^2 n/2C_*$ and hence ${\rm (i)}$ follows.

It remains to prove  ${\rm (ii)}$. Let us first show that  the following procedure (G) is well defined:

\noindent \underline{Input}: A path $\eta \in \sO_*(0^*,(n \e_1)^*)$ satisfying $\clo(\eta) \neq \varnothing$ and $\{ 0^*,(n \e_1)^*\} \not \subset \Lambda_{3R_e}(e) \,\, \forall e \in \mathrm{clo}(\eta) $. 

\noindent \underline{Output}: An edge $e_* =e_*(\eta) \in \clo(\eta)$ and a path $\eta_{e_*} \in \sO_*(0^*,(n \e_1)^*)$ such that 
\begin{itemize}
     \item [(${\rm a}$)]  $e_* \in \arg \max \{R_{e}: e \in \clo(\eta)\}$; 
     \item [(${\rm b}$)]  $\eta_{e_*} \setminus \eta$ is open;
    \item [(${\rm c}$)] $\clo(\eta_{e_*}) \cap \Lambda_{R_{e_*}}({e_*}) = \varnothing$; 
     \item [(${\rm d}$)] $\rmT(\eta_{e_*} \setminus \eta)= |\eta_{e_*} \setminus \eta| \leq 2 C_* R_{e_*}$. 
\end{itemize}
Indeed, we simply take $e_* \in \clo(\eta)$  with the largest effective radius, using a rule to break ties when there are several maximizing edges.   Applying Proposition \ref{Lem: application of er 2} to $\eta \in \sO_{*}(0^*,(n\e_1)^*)$ and $e_*$, we obtain $\eta_{e_*} \in \sO_{*}(0^*,(n\e_1)^*)$, a modified path of $\eta$, satisfying $({\rm b}), ({\rm c})$, $({\rm d})$.
\begin{figure}[htbp]
\begin{center}
\includegraphics[width=16cm]{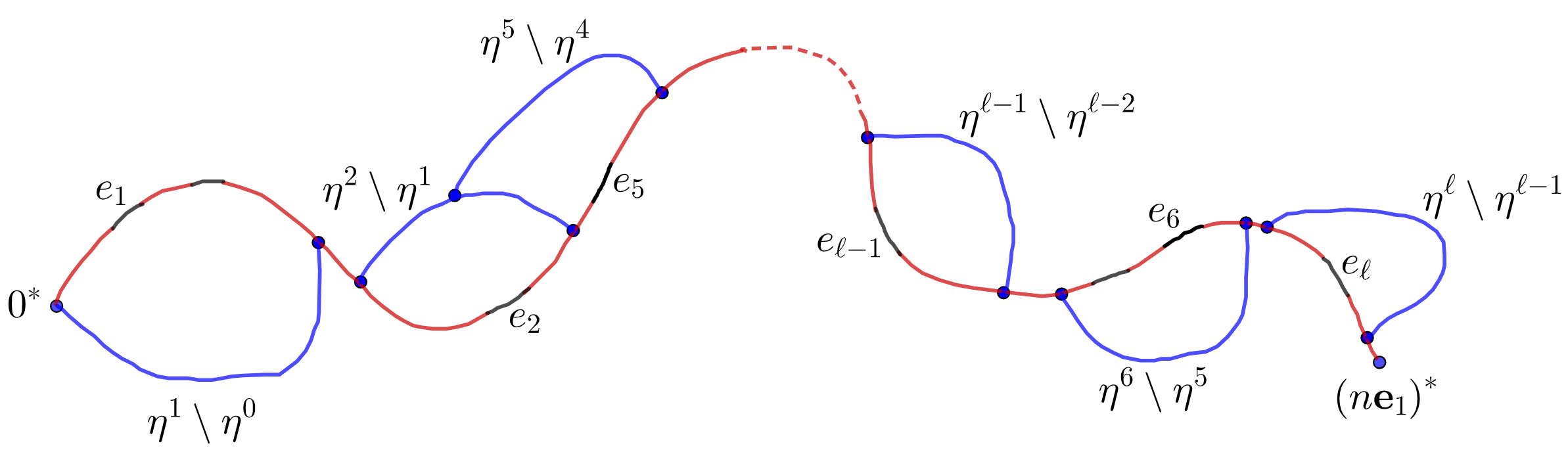}
\caption{Illustration for covering process of all closed edges in the original geodesic $\gamma$. Black lines (resp. red) represent closed edges (resp. open) in $\gamma$, and the open segments $(\eta^{k}\setminus \eta^{k-1})_{k =1}^\ell$ are shown by blue.}
\label{de3}
\end{center}
\end{figure}

We set $\eta^0=\gamma$ and  iteratively apply the procedure (G)  to obtain a sequence of edges $(e_k)_{k=1}^\ell \subset \clo(\eta^0)$ and paths $(\eta^k)_{k=1}^\ell$ such that $e_k=e_*(\eta^{k-1})$, $\eta^k=(\eta^{k-1})_{e_k} \in \sO_*(0^*,(n \e_1)^*)$ for $k=1,\ldots, \ell$, and that ${\rm (a)}$--${\rm (d)}$ hold for each pair of $(e_k,\eta^k)_{k=1}^\ell$ and more importantly the final path $\eta^\ell$ is open (see Figure \ref{de3} for illustration). Here we note that $(e_k)_{k=1}^\ell \subset \clo(\eta^0)$ since $\eta^k \setminus \eta^{k-1}$ is open, for each  $k= 1,2, \ldots, \ell$. Remark further that the iteration stops after some finite steps $\ell$, since  $|\clo(\eta^0)|=|\clo(\gamma)|$ is finite and $|\clo(\eta^k)| < |\clo(\eta^{k-1})|$ using (b) and (c). Setting $\Gamma =\{e_1, \ldots,e_\ell\}$, we can easily deduce ${\rm (iia)}$ and ${\rm (iib)}$ from ${\rm (a)}$--${\rm (d)}$. 
 Indeed, property ${\rm (b)}$ gives $\clo(\eta^{k}) \subset \clo(\eta^{k-1})$, and then use property ${\rm (a)}$, we obtain that the sequence $(R_{e_k})_{k=1}^\ell$ is non-increasing, and hence by ${\rm (c)}$,
\begin{align*}
  \forall 1 \leq i < j \leq \ell,\quad  \|e_i-e_j\|_{\infty} \geq R_{e_{i}} = \max \{R_{e_{i}},R_{e_{j}}\},
\end{align*}
which proves ${\rm (iia)}$. Since $\eta^\ell$ is an open path from $0^*$ to $(ne_1)^*$, property (${\rm d}$) gives
\[ 
    0 \leq \nt - \rmT(0^*, (n\e_1)^*)  \leq  
    \rmT(\eta^{\ell}) -  \rmT(\eta^{0})  \leq \sum_{i=1}^{\ell} \rmT(\eta^{i})- \rmT(\eta^{i-1})  \leq \sum_{i=1}^{\ell}|\eta^{i} \setminus \eta^{i-1}| \leq  2C_* \sum_{e \in \Gamma} R_e.
    \]
This proves ${\rm (iib)}$.
\end{proof}
\noindent \textbf{Proof of Theorem \ref{prodecre}}.
It is enough to prove the following statement: there exists a positive constant $c$ such that for all $L \geq \log^2n,$
\begin{align} \label{goal2.1}
    \pr\left(|\nt - \rmT(0^*, (n\e_1)^*) | \geq L \right) \leq c^{-1} \exp(-c L/\log L).
\end{align}
 Indeed, by \eqref{goal2.1} and Lemma \ref{hole},
\begin{align*}
    \pr(|\rmD_n^* - \rmT_n| \geq L) & \leq \pr(|\rmD_n^* - \rmT(0^*, (n\e_1)^*) | \geq L/2) + \pr( \|0^*\|_1 +\|(n \e_1)^*- n \e_1\|_1 \geq  \tfrac{L}{2\log^2 n}) \\
    &  \leq  c^{-1} \exp(-c L/\log L) + 2 \pr ( \|0^*\|_1 \geq \tfrac{L}{4\log^2 n}) \\
    & \leq  c_1^{-1} \exp\left(-c_1 \tfrac{L}{\log L + \log^2 n }\right),
\end{align*}
where $c_1$ is a positive constant. The remainder of this section is devoted to proving \eqref{goal2.1}.
Let $\gamma_n$ be a geodesic of $\rmT(0^*,(n \e_1)^*)$ with some deterministic rules breaking ties.
By Lemma \ref{hole} and Lemma \ref{Lem: large deviation of graph distance Dlambda}, there exist a positive constant $C$ such that   
\ben{ \label{00s}
  \pr( \rmD^*_n \leq Cn/2; \|0-0^*\|_{\infty}, \|n\e_1-(n\e_1)^*\|_{\infty}  \leq n/4)  \geq 1-C\exp(-n/C).
}
Remark further that $\rmT(0^*,(n \e_1)^*) \leq \rmD^*_n$, and  if $\rmT(0^*,(n \e_1)^*) \leq k$ then $\gamma_n \subset \Lambda_k(0^*)$,  since $t_e \geq 1$ for all $e$. Hence, if the event in \eqref{00s} holds then $\gamma_n \subset \Lambda_{Cn}$, and thus we have
\ben{ \label{pecn}
\pr(\kE) \geq 1-C\exp(-n/C), \qquad \kE:=\{\gamma_n \subset \Lambda_{Cn}\}.
}
We also define
\be{ 
\kE_*:=\{ \{0^*,(n \e_1)^*\} \not \subset \Lambda_{3R_e}(e) \, \forall \, e \in \clo(\gamma_n) \}.
}
Using \eqref{00s}, \eqref{pecn}, and Proposition \ref{Lem: effective radius}, there are positive constants $C$ and $c$ such that
\ban{ \label{esc}
\pp(\kE_*^c) &\leq  \pp(\kE_*^c \cap \kE \cap \{\|0-0^*\|_{\infty} + \|n \e_1- (n \e_1)^*\|_{\infty} \leq n/2\}) + \pp(\kE^c) \notag \\
& \qquad \qquad \qquad \qquad \qquad \qquad\qquad + \pr(\{\|0-0^*\|_{\infty} + \|n \e_1- (n \e_1)^*\|_{\infty} \geq n/2) \notag  \\
& \leq \pp(\exists \, e \in \Lambda_{Cn}: R_e \geq n/12) + 3 C \exp(-cn) \leq 4 C\exp(-cn).
}
{\bf Case 1}: $L \geq Cn$. Using Lemma \ref{Lem: large deviation of graph distance Dlambda},
\be{
  \pr\left(|\nt - \rmT(0^*,(n \e_1)^*)| \geq L \right) \leq \pr\left(\nt \geq L \right)  \leq \exp(-c L),
}
and the result follows. \\

\noindent {\bf Case 2}: $L < Cn$. Using \eqref{pecn}, \eqref{esc}, Propositions \ref{Lem: effective radius} and \ref{lemdis}, 
\begin{align} \label{48}
    \pr & \left(|\nt  -  \rmT(0^*,(n \e_1)^*)|  \geq  L \right) \notag \\
    & \leq \pr (\kE^c) + \pr(\kE_*^c)+ \pr \left( \exists \, e\in \Lambda_{Cn}: R_e \geq L \right)  +\pr \left(\exists \, \Gamma_n \subset \Lambda_{Cn} \textrm{ satisfying (a)--(c)} \right) \notag \\
    & \leq c^{-1} \exp(-cL) +\pr \left(\exists \, \Gamma_n \subset \Lambda_{Cn} \textrm{ satisfying (a)--(c)} \right),
\end{align}
where $c$ is a positive constant and
\begin{itemize}
\item [(a)]  $ \frac{\log^2 n}{2C_*} \leq R_e \leq L$ for all $e \in \Gamma_n$,
    \item [(b)] 
    $ \|e-e'\|_{\infty} \geq \max\{R_e,R_{e'}\}$,  for all $e \neq e' \in \Gamma_n $,
    \item [(c)]   $\sum_{e \in \Gamma_n} R_e \geq L/2C_*$.  
\end{itemize}
In order to estimate the last term  of \eqref{48}, we  set $M_0 = \lceil \log^{3/2} n \rceil$, and $M_q = M_0 2^q$ for $1 \leq q \leq a_n$ with $a_n = \lceil \log_2 L -\log_2 \log n \rceil$. Remark that $\log^{3/2} n \leq  R_e \leq L \leq M_{a_n}$ for all $e \in \Gamma_n$ and thus,
\begin{align*}
  \sum_{e \in \Gamma_n}  R_{e} \leq   \sum_{q=0}^{a_n} 2 M_q N_q,
\end{align*}
 where for each $0 \leq q \leq a_n$,
\begin{align*}
    N_q: = |\{ e \in \Gamma_n: R_e \in [M_q,2M_q] \}|. 
\end{align*}
Therefore, if (c) occurs then there exists $0 \leq q \leq a_n $ such that $N_q \geq B_q$  where for each  $0 \leq q \leq a_n$,
$$ B_q:= \Big \lfloor \dfrac{L}{4C_* M_q  \log_2 L} \Big\rfloor.$$ 
Hence, it follows from the  union bound that
\begin{align} \label{<2/3}
   &\pr \left(\exists \, \Gamma_n \subset \Lambda_{Cn} \textrm{ satisfying (a)--(c)} \right) \leq \sum_{q=0}^{a_n} \pr \left(\exists \, \Gamma_n \subset \Lambda_{Cn} \textrm{ satisfying (a), (b)}; \, N_q \geq B_q \right).
\end{align}
Moreover, by (b), if $R_e$ and $R_{e'}$ are in $[M_q,2M_q]$ for some $e \neq e' \in \Gamma_n$, then 
\ben{ \label{hkmq}
\|e -e'\|_\infty \geq \max\{R_{e}, R_{e'}\}\geq  M_q.
}
Therefore,
\begin{align}\label{nqbound}
   \pr \left(\exists \, \Gamma_n \subset \Lambda_{Cn} \textrm{ satisfying (a), (b)}; \, N_q \geq B_q \right)\leq S_q,
\end{align}
where 
\begin{align} \label{sq1}
 S_q : = & \pr \Big( \, \exists \,\{ e_1,\ldots,e_{B_q} \} \subset \Lambda_{Cn}:  R_{e_j} \in [M_q,2M_q] \,\forall \,  j \in [ B_q]; 
   \|e_j-e_k\|_\infty \geq M_q \, \forall   j \neq k \in [B_q] \Big),
\end{align}
here recall the notation $[a]=\{1,\ldots,a\}$. The following claim is straightforward.\\

\noindent \textbf{Claim}. \textit{Given $C$, there exists a constant $c=c(d,C) >0$ such that the following holds: for any $M \in \N$ and $E \subset \cE$ satisfying $$\|e-e'\|_{\infty} \geq M, \quad \forall e \neq e' \in E,$$
 we can find $E' \subset E$ such that $|E'| \geq c|E| $ and $\|e-e'\|_{\infty} \geq 5CM$ for all $e \neq e' \in E'$.} \\

 By this claim, there exists a positive constant $c$  such that  for any $0 \leq q \leq a_n$ if the event in \eqref{sq1} occurs then we can find $E' \subset \{e_1,\ldots, e_{B_q} \}$ such that $|E'| \geq  \lfloor c B_q\rfloor $ and $\|e -e'\|_{\infty} \geq 5C_*M_q$  for all $e \neq e' \in E'$. As a result, we have
\begin{align} \label{sq}
    S_q & \leq  \pr \left( \, \exists \, \{e'_1,\ldots,e'_{C_q}\} \in \mathcal{T}_q: R_{e'_j} \in [M_q,2M_q] \, \forall   j \in [C_q] \right)  \notag \\
  & \leq \sum_{\{e'_1,\ldots,e'_{C_q} \} \in \mathcal{T}_q} \pr \left(R_{e'_j} \in [M_q ,2M_q] \, \forall j \in [C_q]\right),
\end{align}
where  $C_q=\lfloor cB_q \rfloor$ and
\begin{align*}
    \mathcal{T}_q = \{\{e'_1,\ldots,e'_{C_q}\} \subset \Lambda_{Cn}:\|e'_j-e'_k\|_{\infty} \geq 5C_*M_q \, \forall j \neq k \in [C_q]\}.
\end{align*} 
 We remark that the event $R_e \in [M_q,2M_q]$ only depends on the state  of edges in $\Lambda_{2C_*M_q}(e)$. Therefore, given  $\{e'_1,\ldots,e'_{C_q}\} \in \mathcal{T}_q$, the family of the events $(\{R_{e'_j} \in [M_q, 2M_q]\})_{i \in [C_q]}$ are independent. Hence,
\begin{align*}
     \pr \left(R_{e'_j} \in [M_q ,2M_q] \, \forall i \in [C_q]\right) &= \prod_{j=1}^{ C_q} \pr \left(R_{e'_j} \in [M_q ,2M_q]\right) \leq \exp(-\alpha M_q C_q),
\end{align*}
where  $\alpha$ is a positive constant. Here for the inequality we have used Proposition \ref{Lem: effective radius} with the remark that $M_q \leq M_{a_n} \leq  n^2$. This estimate together with  \eqref{sq} yields  for all $0 \leq q \leq a_n$,
\begin{align*}
    S_q  \leq  |\mathcal{T}_q| \exp(-\alpha  C_qM_q)  &\leq    (4Cn)^{ d C_q}  \exp(-\alpha C_q M_q)    \leq  \exp(-\alpha C_q M_q /2)  \leq   \exp\left(- \tfrac{c \alpha L}{16 C_* \log_2 L}\right).
\end{align*}
Combining the above estimate with   \eqref{nqbound}, \eqref{<2/3}, and \eqref{48}, we obtain  \eqref{goal2.1}.  \qed 
\bigskip

\noindent{\bf Acknowledgment.} \rm
We would like to thank Shuta Nakajima for his valuable comments. This work  is funded by Vietnam National Foundation for Science and Technology De
velopment (NAFOSTED) under grant number 101.03-2023.34. V. Q. Nguyen was supported by the
Master, PhD Scholarship Programme of Vingroup Innovation Foundation (VINIF), code VINIF.2024.TS.002. We also  thank the Vietnam Institute for Advanced
Study in Mathematics (VIASM) for its hospitality and the fruitful research environment, where
part of this work was carried out.
\bigskip
\bibliographystyle{amsplain}
\bibliography{citation} 

\providecommand{\bysame}{\leavevmode\hbox to3em{\hrulefill}\thinspace}
\providecommand{\MR}{\relax\ifhmode\unskip\space\fi MR }
\providecommand{\MRhref}[2]{%
  \href{http://www.ams.org/mathscinet-getitem?mr=#1}{#2}
}
\providecommand{\href}[2]{#2}
\begin{thebibliography}{10}

\bibitem{ahlberg2023chaos}
Daniel Ahlberg, Maria Deijfen, and Matteo Sfragara, \emph{Chaos, concentration and multiple valleys in first-passage percolation}, arXiv preprint arXiv:2302.11367 (2023).

\bibitem{alexander2013subgaussian}
Kenneth Alexander and Nikolaos Zygouras, \emph{Subgaussian concentration and rates of convergence in directed polymers}, Electronic Journal of Probability \textbf{18} (2013), 1--28.

\bibitem{antal1996chemical}
Peter Antal and Agoston Pisztora, \emph{On the chemical distance for supercritical bernoulli percolation}, The Annals of Probability \textbf{24} (1996), no.~2, 1036--1048.

\bibitem{auffinger201750}
Antonio Auffinger, Michael Damron, and Jack Hanson, \emph{50 years of first-passage percolation}, vol.~68, American Mathematical Soc., 2017.

\bibitem{benaim2008exponential}
Michel Bena{\"\i}m and Rapha{\"e}l Rossignol, \emph{Exponential concentration for first passage percolation through modified poincar{\'e} inequalities}, Annales de l'IHP Probabilit{\'e}s et statistiques \textbf{44} (2008), no.~3, 544--573.

\bibitem{benjamini2003improved}
I~Benjamini, G~Kalai, and O~Schramm, \emph{First passage percolation has sublinear distance variance}, Ann. Probab \textbf{31} (2003), no.~4, 1970--1978.

\bibitem{bernstein2020sublinear}
Megan Bernstein, Michael Damron, and Torin Greenwood, \emph{Sublinear variance in euclidean first-passage percolation}, Stochastic Processes and their Applications \textbf{130} (2020), no.~8, 5060--5099.

\bibitem{broadbent1957percolation}
Simon~R Broadbent and John~M Hammersley, \emph{Percolation processes: I. crystals and mazes}, Mathematical proceedings of the Cambridge philosophical society \textbf{53} (1957), no.~3, 629--641.

\bibitem{nakajima2019first}
Van~Hao Can and Shuta Nakajima, \emph{First passage time of the frog model has a sublinear variance}, Electronic Journal of Probability \textbf{24} (2019), 1--27.

\bibitem{can2023lipschitz}
Van~Hao Can, Shuta Nakajima, and Van~Quyet Nguyen, \emph{Lipschitz-continuity of time constant in generalized first-passage percolation}, Stochastic Processes and their Applications \textbf{175} (2024), 104402.

\bibitem{cerf2022time}
Rapha{\"e}l Cerf and Barbara Dembin, \emph{The time constant for bernoulli percolation is lipschitz continuous strictly above pc}, The Annals of Probability \textbf{50} (2022), no.~5, 1781--1812.

\bibitem{cerf2016weak}
Rapha{\"e}l Cerf and Marie Th{\'e}ret, \emph{Weak shape theorem in first passage percolation with infinite passage times}, Annales de l'Institut Henri Poincar{\'e}, Probabilit{\'e}s et Statistiques \textbf{52} (2016), no.~3, 1351--1381.

\bibitem{chatterjee2014superconcentration}
Sourav Chatterjee, \emph{Superconcentration and related topics}, vol.~15, Springer, 2014.

\bibitem{chatterjee2023superconcentration}
\bysame, \emph{Superconcentration in surface growth}, Random Structures \& Algorithms \textbf{62} (2023), no.~2, 304--334.

\bibitem{cox1980time}
J~Theodore Cox, \emph{The time constant of first-passage percolation on the square lattice}, Advances in Applied Probability \textbf{12} (1980), no.~4, 864--879.

\bibitem{cox1981continuity}
J~Theodore Cox and Harry Kesten, \emph{On the continuity of the time constant of first-passage percolation}, Journal of Applied Probability \textbf{18} (1981), no.~4, 809--819.

\bibitem{damron2014subdiffusive}
Michael Damron, Jack Hanson, and Philippe Sosoe, \emph{Subdiffusive concentration in first passage percolation}, Electronic Journal of Probability \textbf{19} (2014), 1--27.

\bibitem{damron2015sublinear}
\bysame, \emph{Sublinear variance in first-passage percolation for general distributions}, Probability Theory and Related Fields \textbf{163} (2015), no.~1, 223--258.

\bibitem{dembin2022variance}
Barbara Dembin, \emph{The variance of the graph distance in the infinite cluster of percolation is sublinear}, ALEA, Lat. Am. J. Probab. Math. Stat. \textbf{21} (2024), 307--320.

\bibitem{dembin2024superconcentration}
Barbara Dembin and Christophe Garban, \emph{Superconcentration for minimal surfaces in first passage percolation and disordered ising ferromagnets}, Probability Theory and Related Fields (2024), 1--28.

\bibitem{dembin2022upper}
Barbara Dembin and Shuta Nakajima, \emph{On the upper tail large deviation rate function for chemical distance in supercritical percolation}, arXiv preprint arXiv:2211.02605 (2022).

\bibitem{duminil2018sixty}
Hugo Duminil-Copin, \emph{Sixty years of percolation}, Proceedings of the International Congress of Mathematicians: Rio de Janeiro 2018, vol.~IV, World Scientific, 2018, pp.~2829--2856.

\bibitem{falik2007edge}
Dvir Falik and Alex Samorodnitsky, \emph{Edge-isoperimetric inequalities and influences}, Combinatorics, Probability and Computing \textbf{16} (2007), no.~5, 693--712.

\bibitem{garet2004asymptotic}
Olivier Garet and R{\'e}gine Marchand, \emph{Asymptotic shape for the chemical distance and first-passage percolation on the infinite bernoulli cluster}, ESAIM: Probability and Statistics \textbf{8} (2004), 169--199.

\bibitem{garet2007large}
\bysame, \emph{Large deviations for the chemical distance in supercritical bernoulli percolation}, The Annals of Probability \textbf{35} (2007), no.~3, 833--866.

\bibitem{garet2009moderate}
\bysame, \emph{Moderate deviations for the chemical distance in bernoulli percolation}, ALEA, Lat. Am. J. Probab. Math. Stat. \textbf{7} (2009), 171--191.

\bibitem{garet2017continuity}
Olivier Garet, R{\'e}gine Marchand, Eviatar~B Procaccia, and Marie Th{\'e}ret, \emph{Continuity of the time and isoperimetric constants in supercritical percolation}, Electronic Journal of Probability \textbf{22} (2017), 1--35.

\bibitem{grimmett1999percolation}
Geoffrey Grimmett, \emph{Percolation}, Springer-Verlag, 1989.

\bibitem{kesten1986aspects}
Harry Kesten, \emph{Aspects of first passage percolation}, {\'E}cole d'{\'e}t{\'e} de probabilit{\'e}s de Saint Flour XIV-1984, Springer, 1986, pp.~125--264.

\bibitem{pisztora1996surface}
Agoston Pisztora, \emph{Surface order large deviations for ising, potts and percolation models}, Probability Theory and Related Fields \textbf{104} (1996), 427--466.

\end{thebibliography}

\end{document}